\documentclass[11pt]{article}
\usepackage{amssymb}
\usepackage{amsthm}
\usepackage{amsmath,enumerate}
\usepackage{fullpage,algorithmic}
\usepackage{verbatim,graphicx,url,amsmath,amsthm,amsfonts,amssymb,subfigure,bbm}
\usepackage{tikz}
\definecolor{darkgreen}{rgb}{0.1,0.6,0.1}
\usepackage[pagebackref,letterpaper=true,colorlinks=true,pdfpagemode=none,linkcolor=blue,citecolor=darkgreen,pdfstartview=FitH]{hyperref}
\usepackage{pdfsync}

\newtheorem{theorem}{Theorem}[section]
\newtheorem{lemma}[theorem]{Lemma}

\newtheorem{claim}[theorem]{Claim}

\newtheorem{corollary}[theorem]{Corollary}

\newtheorem{observation}[theorem]{Observation}

\newcommand{\pr}{\mathbb{P}}

\newcommand{\jnote}[1]{}

\newcommand{\E}{{\mathbb E}}
\newcommand{\diam}{\mathrm{diam}}

\newcommand{\Lip}{\mathrm{Lip}}
\newcommand{\dem}{\mathsf{dem}}
\newcommand{\cp}{\mathsf{cap}}
\newcommand{\flow}{\varphi}

\newcommand{\dist}{\mathsf{dist}}

\newcommand{\remove}[1]{}
\newcommand{\1}{\mathbf{1}}

\newcommand{\len}{\mathsf{len}}

\newcommand{\e}{\varepsilon}

\newcommand{\minB}{{\textrm {hit}_B}}
\newcommand{\maxB}{{\textrm{cov}_B}}

\newcommand{\mcf}{\mathsf{mcf}}

\newcommand{\ext}{\mathsf{ext}}

\begin{document}

\title{{\bf A node-capacitated Okamura-Seymour theorem}}
\author{James R. Lee\thanks{Computer Science \& Engineering, University of Washington.  Research supported by NSF CCF-1217256 and a Sloan Research Fellowship.} \and
Manor Mendel\thanks{Mathematics and Computer Science Department, Open University of Israel.  Research partially supported by ISF grants 221/07, 93/11,
and BSF grants 2006009, 2010021.
This work was carried out during visits of the author to Microsoft Research
and the University of Washington.} \and Mohammad Moharrami$^*$}
\date{}

\maketitle

\begin{abstract}
The classical Okamura-Seymour theorem states that
for an edge-capacitated, multi-commodity flow instance in which all terminals
lie on a single face of a planar graph, there exists a feasible
concurrent flow if and only if the cut conditions are satisfied.
Simple examples show that a similar theorem is impossible in the node-capacitated setting.
Nevertheless, we prove that an approximate flow/cut theorem does hold:
For some universal $\varepsilon > 0$, if the node cut conditions are satisfied,
then one can simultaneously route an $\varepsilon$-fraction of all the demands.
This answers an open question of Chekuri and Kawarabayashi.
More generally, we show that this holds in the setting of multi-commodity polymatroid networks
introduced by Chekuri, et. al.
Our approach employs a new type of random metric embedding
in order to round the convex programs corresponding to these more general flow problems.
\end{abstract}

\setcounter{tocdepth}{2} {\small \tableofcontents}

\section{Introduction}

The relationship between flows and cuts in graphs has played a fundamental role in
combinatorial optimization.
A seminal result in the study of multi-commodity flows is
the classical Okamura-Seymour theorem \cite{OS81}
which we now recall.

An {\em undirected flow network} is an undirected graph $G=(V,E)$ together with a
capacity function  on edges
$\cp : E \to [0,\infty)$.  A set of {\em demands} is specified by a
symmetric mapping $\dem : V \times V \to \mathbb [0,\infty)$.
For $u,v\in V$, denote by $\flow_{uv}:E\to [0,\infty)$ the undirected $u$-$v$ flow.
The (edge) capacity constrains require that for every $e\in E$, $\sum_{u,v \in V} \flow_{uv}(e)\le \cp(e)$.
Given such an instance, let $\mcf_G(\cp,\dem)$ be the largest value
$\e$ such that one can simultaneous
route $\e \cdot \dem(u,v)$ units of flow between $u$ and $v$
for every $u,v \in V$ while not violating any of the edge capacities.
This optimization describes the {\em maximum concurrent flow problem.}

For two subsets $S,T \subseteq V$, let $\cp(S,T)$ denote the total
capacity of all edges with one endpoint in $S$ and one in $T$.
Similarly, let $\dem(S,T) = \sum_{u \in S} \sum_{v \in T} \dem(u,v)$.
To give an upper bound on $\mcf$, we can consider cuts in $G$,
described by subsets $S \subseteq V$.  To every such subset we assign a value
called the {\em sparsity} of the cut:
$$
\Phi_G(S; \cp, \dem) = \frac{\cp(S, \bar S)}{\dem(S,\bar S)}\,.
$$
It is straightforward to see that for any $S \subseteq V$, we have $\mcf_G(\cp,\dem) \leq \Phi_G(S;\cp,\dem)$.
The {\em sparsest cut} is the one which gives the best upper bound on $\mcf$.  In this vein, we define
$$
\Phi_G(\cp,\dem) = \min_{S \subseteq V} \Phi_G(S;\cp,\dem)\,.
$$
Thus we have the relationship $\mcf_G(\cp,\dem) \leq \Phi_G(\cp,\dem)$ and
the flow/cut gap question asks how close this upper bound is to the truth.

To state the Okamura-Seymour theorem, we need one final piece of notation.
We say that the demand function $\dem$ is {\em supported on a subset $D \subseteq V$}
if $\dem(u,v) > 0$ only when $u,v \in D$.  The classical Max-flow Min-cut Theorem \cite{FF56}
implies that if the demand $\dem$ is supported on a two-element subset $\{s,t\} \subseteq V$, then
for any capacities $\cp$, we have $\mcf_G(\cp,\dem) = \Phi_G(\cp,\dem)$.
An extension of Hu \cite{Hu63} shows that if $\dem$ is supported on a 4-element subset $D \subseteq V$,
the same equality holds.
The Okamura-Seymour theorem states that whenever $G$ is a planar
graph and the demand is supported on a single face, there
is likewise no flow/cut gap.

\begin{theorem}[\cite{OS81}]
\label{thm:OS}
Let $G=(V,E)$ be a planar graph, and let $F \subseteq V$ be any face of $G$.
Then for any capacities $\cp :E \to [0,\infty)$ and any demands $\dem : V \times V \to [0,\infty)$
supported on $F$, we have
$$
\mcf_G(\cp,\dem) = \Phi_G(\cp,\dem)\,.
$$
\end{theorem}

We remark that this theorem has applications beyond flow/cut gaps.
For instance, in \cite{CKS09} the authors use it as a fundamental
step in solving the edge-disjoint paths problem in planar graphs
with constant congestion.
A significant motivation of the present paper is to
serve as a step in extending their work
to the more difficult vertex-disjoint paths problem.

Indeed, one can consider generalizations of edge-capacitated networks.  A prominent example
is to consider capacities on vertices (one motivation is that this
a more relevant type of constraint in wireless networks \cite{CKRV12}).
While one can simulate edge capacities by
introducing a new vertex in the middle of an edge, it does not seem that
any reduction is known by which one can simulate vertex capacities with edge capacities.

\medskip
\noindent
{\bf Vertex-capacitated flows.}
Formally, we define a {\em vertex-capacitated flow network} by considering
a function $\cp~:~V~\to~[0,\infty)$ assigning capacities to vertices
instead of edges.  It seems that the most elegant way to think about
capacities in this setting is as follows:  If a flow of value $\alpha$ is sent along a path $P$ from $s$ to $t$,
then it consumes $\alpha/2$ capacity at $s$ and $t$ and $\alpha$ capacity at each of the intermediate nodes of $P$.\footnote{This
particular choice does not materially affect any theorem in the paper which deals with approximate flow/cut gaps.}
Formally, in the multi-commodity setting, the vertex capacity constrains require that for every $w\in V$,
$\sum_{e \ni w} \sum_{u,v \in V} \flow_{uv}(e)\le 2\,\cp(w).$
The corresponding definition of the maximum concurrent flow follows immediately; we use the notation $\mcf_G^v$
for the vertex-capacitated version.  For the definition of $\Phi_G^v$, we have to be slightly more careful.
For a subset $S\subseteq V$ of the vertices, denote by $G[S]$ the induces subgraph of $G$ on $S$.
We define a function $\rho_S : V \times V \to \{0,\frac12,1\}$ by
$$
\rho_S(u,v) = \begin{cases}
\frac12 & |\{u,v\} \cap S| = 1 \\
1 & u,v \in S \\
1 & u,v \in \bar S \textrm{ and $u,v$ are in distinct connected components of $G[\bar S]$} \\
0 & \textrm{otherwise.}
\end{cases}
$$
In other words, we are only given half-credit for separating $u$ and $v$ if exactly
one of them is in the separator.
Then we define $$\Phi_G^v(S; \cp, \dem) = \frac{\sum_{v \in S} \cp(v)}{\sum_{u,v \in V} \dem(u,v) \rho_S(u,v)}\,,$$
and $\Phi_G^v(\cp,\dem) = \min_{S \subseteq V} \Phi_G^v(S; \cp,\dem)$.
It is straightforward to verify that $\mcf_G^v(\cp,\dem) \leq \Phi_G^v(\cp,\dem)$.

\begin{figure}
\begin{center}
\includegraphics[width=3.8in]{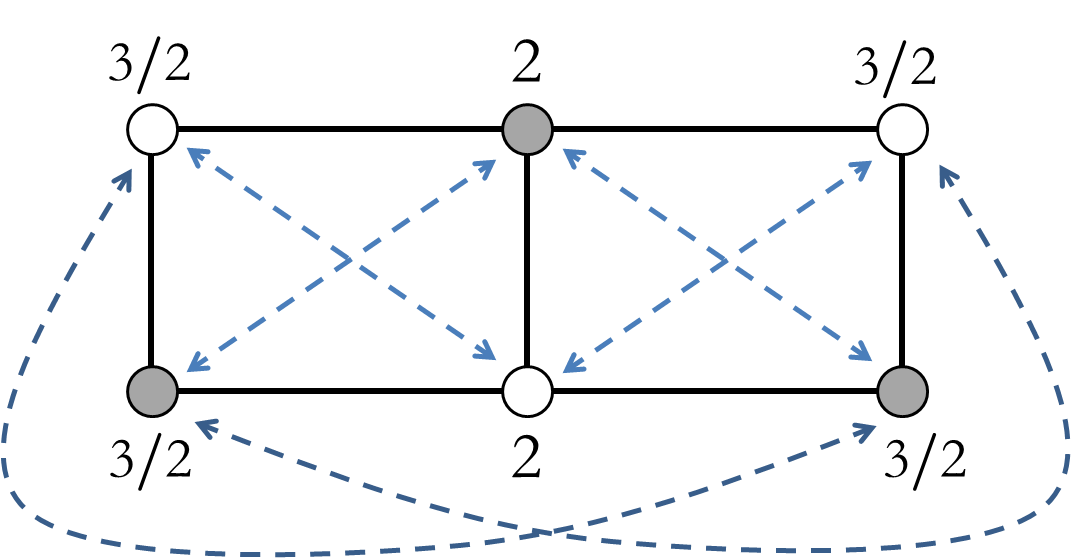}
\caption{A counterexample to an exact node-capacitated Okamura-Seymour Theorem.\label{fig:ce}}
\end{center}
\end{figure}

These precise definitions ensure that a classical Max-flow Min-cut theorem holds when the demand is supported
on a single pair (this follows from Menger's theorem \cite{Menger27}).
They also allow other natural properties in the multi-commodity setting; it is an exercise to show that
for any tree $T$, we have $\mcf^v_T(\cp,\dem)=\Phi_T^v(\cp,\dem)$
for any choice of capacities and demands.  In some sense,
this property will be crucial for our approach later.

Unfortunately, there is no exact
vertex-capacitated analog of the Okamura-Seymour Theorem.  The planar graph in Figure \ref{fig:ce}
has all vertices on the outer face.  The capacities are specified on the vertices
and the demands are given by dotted edges in the figure; all demands have value 1.
It is straightforward to check that one has $\Phi_G^v(\cp,\dem)=1$
and yet $\mcf_G^v(\cp,\dem) = 5/7$.  The shaded nodes form a vertex cut of sparsity 1.

Nevertheless, a main result of the present paper is that an approximate
version does hold in the vertex-capacitated setting,
answering a question posed by Chekuri and Kawarabayashi.

\begin{theorem}[Approximate Okamura-Seymour Theorem]
\label{thm:nodeoks}
There exists a constant $\e > 0$ such that the following holds.
Let $G=(V,E)$ be a planar graph and let $F \subseteq V$ be any face of $G$.
Then for any vertex capacities $\cp : V \to [0,\infty)$ and any demands $\dem : V \times V \to [0,\infty)$
supported on $F$, we have
$$
\mcf^v_G(\cp,\dem) \geq \e \cdot \Phi^v_G(\cp,\dem)\,.
$$
\end{theorem}

In fact, our result holds in the more general setting of undirected polymatroid networks which we discuss next.

\subsection{Polymatroid networks}

Motivated by applications to information flow in wireless networks, Chekuri et.~al.~\cite{CKRV12}
introduced a generalization of vertex capacities by putting a submodular capacity function
at every vertex.
Recall that a function
$f : 2^S \to \mathbb R$ over a finite set $S$ is called {\em submodular}
if $f(A) + f(B) \geq f(A \cap B) + f(A \cup B)$ for all $A,B \subseteq S$.
 Let $G=(V,E)$ be a graph and suppose that for a multi-commodity flow $\flow=\{\flow_{st}\}_{s,t\in V}$
in $G$, we use $\varphi(e)=\sum_{s,t} \varphi_{st}(e)$ to denote the total flow through the edge $e$.
For a vertex $v \in V$, we use $E(v)$ to denote the edges incident to $v$.
Let $\vec \rho = \{\rho_v : 2^{E(v)} \to [0,\infty)\}_{v \in V}$
be a collection of monotone, submodular
 functions called {\em polymatroid capacities.}
A flow $\varphi$ is {\em feasible with respect to $\vec \rho$} if
for every $v \in V$ and every subset $S \subseteq E(v)$, we have
$\sum_{e \in S} \varphi(e) \leq \rho_v(S)$.

Given a demand function $\dem : V \times V \to [0,\infty)$, we can define the {\em maximum concurrent flow
value} of a polymatroid network by $\mcf_G(\vec \rho, \dem)$ as the maximum $\e > 0$
such that one can route an $\e$-fraction of all demands simultaneously using
a flow that is feasible with respect to $\vec \rho$.

The corresponding notion
of a sparse cut is now a little trickier.
For every subset of edges $S \subseteq E$, we can define the cut semi-metric $\sigma_S : V \times V \to \{0,1\}$ on $V$ by $\sigma_S(x,y) = 0$
if and only if there exists a path from $x$ to $y$ in the graph $G(V, E \setminus S)$.
Following \cite{CKRV12}, we call a map $g : S \to V$ {\em valid} if it maps every
edge in $S$ to one of its two endpoints in $V$.  We can then define the {\em capacity of a set $S \subseteq E$} by
$$
\nu_{\vec \rho}(S) = \min_{\substack{g : S \to V \\ \textrm{valid}}} \sum_{v \in V} \rho_v(g^{-1}(v))\,.
$$
Finally, we define the {\em sparsity of $S$} by
$$
\Phi_G(S; \vec \rho, \dem) = \frac{\nu_{\vec \rho}(S)}{\sum_{u,v \in V} \dem(u,v) \sigma_S(u,v)}\,,
$$
and define $\Phi_G(\vec \rho,\dem) = \min_{S \subseteq V} \Phi_G(S;\vec \rho,\dem).$
It is not too difficult to see that, again,
$$
\mcf_G(\vec \rho, \dem) \leq \Phi_G(\vec \rho,\dem)\,.
$$
In \cite{CKRV12}, it is proved that when $\dem$ is supported on a single pair, we have
$$
\Phi_G(\vec \rho,\dem) \leq 2 \cdot \mcf_G(\vec \rho, \dem)\,.
$$

Unfortunately, the factor 2 is necessary, and owes itself to a slight
defect in the notion of undirected polymatroid networks.  If one were to
say that a flow only consumes half the capacity of an edge
if it originates at an endpoint (as in the vertex-capacitated case
described above), then we would obtain an exact single-commodity max-flow/min-cut theorem
in this setting.  Indeed, for directed polymatroid networks,
such a result is classical \cite{Hassin78,LM82}.
Since we are concerned here with approximate flow/cut gaps, this
will not be an issue, and we follow \cite{CKRV12}.
We obtain an Okamura-Seymour theorem for polymatroid networks as well, answering a question
posed to us by Chandra Chekuri.

\begin{theorem}[Polymatroid Okamura-Seymour Theorem]
\label{thm:polyoks}
There exists a constant $\e > 0$ such that the following holds.
Let $G=(V,E)$ be a planar graph and let $F \subseteq V$ be any face of $G$.
Then for any polymatroid capacities $\vec \rho$ and any demands $\dem : V \times V \to [0,\infty)$
supported on $F$, we have
$$
\mcf_G(\vec \rho,\dem) \geq \e \cdot \Phi_G(\vec \rho,\dem)\,.
$$
\end{theorem}

Theorem~\ref{thm:nodeoks} is a special case of Theorem~\ref{thm:polyoks}.
Indeed, vertex capacity $\cp:V\to [0,\infty)$, is (up to a factor of 2) equivalent to vertex polymatroid capacity
\( \vec \rho_v(\emptyset)=0 \) and $\vec\rho_v(S)=\cp(v)$ for $\emptyset\ne S\subseteq E(v)$. With this definition of $\vec \rho$, it is immediate to check that
$\mcf_G(\vec \rho,\dem) \le \mcf_G^v(\cp,\dem) \le 2 \mcf_G(\vec \rho,\dem)$
and
$\Phi_G(\vec \rho,\dem) \le \Phi_G^v(\cp,\dem) \le 2 \Phi_G(\vec \rho,\dem)$.

\subsection{Embeddings and flow/cut gaps}

Our main tools in proving Theorems \ref{thm:nodeoks} and \ref{thm:polyoks}
are various embeddings of metric spaces.  To this end, we first recall
known results in the edge and vertex-capacitated settings.
In the next section, we discuss the new types of embeddings
we need to handle vertex-capacitated and polymatroid networks.

A {\em metric graph} $G=(V,E,\len)$ is an undirected
graph equipped with a non-negative length function on edges $\len : E \to [0,\infty)$.
We extend the length function to paths $P \subseteq E$ by setting $\len(P) = \sum_{e \in P} \len(e)$.
Associated to every such length is the shortest-path pseudo-metric on $G$
defined by $d_{\len}(u,v) = \min_P \len(P)$ where the minimum is
over all $u$-$v$ paths $P$ in $G$.  We say that a pseudo-metric $d$ on $V$ is {\em supported
on the graph $G$} if $d=d_{\len}$ for some length function on $E$.
In many situations we will only be considering a single length function on $G$ at a time,
and then we write $d_G$ instead of $d_{\len}$.

We will consider embeddings of such graph metrics into various other spaces.
Given two metric spaces $(X,d_X)$ and $(Y,d_Y)$ and a function $f : X \to Y$,
we define the {\em Lipschitz constant of $f$} by
$$
\|f\|_{\Lip} = \sup_{x \neq y \in X} \frac{d_Y(f(x),f(y))}{d_X(x,y)}\,.
$$
If $\|f\|_{\Lip} \leq L$, we say that $f$ is {\em $L$-Lipschitz.}

We define the {\em distortion} of the map $f$ by $\dist(f) = \|f\|_{\Lip} \cdot \|f^{-1}\|_{\Lip}$.
The {\em $L_1$ distortion of a metric space $(X,d_X)$,} written $c_1(X,d_X)$,
denotes the infimum of $\dist(f)$ over all maps $f : X \to L_1$.
The next theorem gives a tight relationship between flow/cut gaps in graphs
and $L_1$ embeddings of the metric supported on them.
It follows from \cite{LLR95} and \cite{GNRS04}.

\begin{theorem}
Consider any graph $G=(V,E)$ and any subset $D \subseteq V$.
Let
$$
K_1(G,D) = \sup_{\cp,\dem} \frac{\Phi_G(\cp,\dem)}{\mcf_G(\cp,\dem)}\,,
$$
where the supremum is over all capacity functions $\cp : E \to [0,\infty)$ and
all demand functions supported on $D$.  Let
$$
K_2(G,D) = \sup_{d} \left[\inf_{f : (V,d) \to L_1} \dist(f|_D)\right],
$$
where the supremum is over all metrics $d$ supported on $G$
and the infimum is over all 1-Lipschitz mappings $f : V \to L_1$.
Then $K_1(G,D) = K_2(G,D)$.
\end{theorem}

In particular, the Okamura-Seymour Theorem (Thm. \ref{thm:OS}) can be restated as the following
fact about embeddings of planar graphs:  For any metric planar graph $G=(V,E)$
and any face $F \subseteq V$, there exists a 1-Lipschitz mapping $f : V \to L_1$
such that $\dist(f|_F)=1$.

\medskip
\noindent
{\bf Vertex-capacitated flows and $\ell_1^{\mathrm{dom}}$ embeddings.}
Unfortunately, $L_1$ embeddings are not sufficient for the study
of vertex-capacitated flow/cut gaps; we refer to \cite{FHL08} for
some examples.
Instead, \cite{FHL08} uses a stronger notion of embedding.
For simplicity, we discuss such embeddings
only for finite metric spaces.
An {\em $\ell_1^{\mathrm{dom}}$ embedding of a finite pseudometric space $(X,d)$}
is a random 1-Lipschitz mapping $\Lambda : X \to \mathbb R$.
One then defines
$$
\dist(\Lambda) = \max_{x, y \in X} \frac{\mathbb E\,|\Lambda(x)-\Lambda(y)|}{d(x,y)}\,,
$$
and writes $c_1^{\mathrm{dom}}(X,d)$ for the infimum of $\dist(\Lambda)$ over
all such random mappings $\Lambda : X \to \mathbb R$.
It is straightforward to verify that $c_1(X,d) \leq c_1^{\mathrm{dom}}(X)$
and there are many interesting cases when this inequality is strict (see \cite{FHL08,BKL07}).
Such embeddings were initially studied by Matousek and Rabinovich \cite{MR03}.
It was shown in \cite{FHL08} that they can be used to bound
vertex-capacitated flow/cut caps, and \cite{CKRV12} extended this to undirected polymatroid networks.

\begin{theorem}[\cite{CKRV12}]
Consider a graph $G=(V,E)$ and a subset $D \subseteq V$.  Suppose
there is a constant $K \geq 1$ such that for
every metric $d$ supported on $G$, we have $c_1^{\mathrm{dom}}(D,d) \leq K$.
Then for every set of polymatroid capacities $\vec \rho$ on $G$
and every $\dem : V \times V \to [0,\infty)$ supported on $D$,
we have
$$
\mcf_G(\vec \rho,\dem) \geq \frac{1}{2K} \Phi_G(\vec \rho,\dem)\,.
$$
\end{theorem}

Despite the power of the preceding theorem,
it is insufficient for proving our main results.
Since $c_1^{\mathrm{dom}}(X,d)$ is at least
the Euclidean distortion of $(X,d)$,
Bourgain's lower bound on the Euclidean distortion of trees~\cite{Bour86}
implies that there are $n$-point tree metrics $(T_n,d_n)$
with $c_1^{\mathrm{dom}}(T_n,d_n) = \Omega(\sqrt{\log \log n})$.
In the next section, we introduce a new notion of embedding
that is sufficient for proving vertex-capacitiated and
polymatroid versions of the Okamura-Seymour theorem.

\subsection{Length functions, star-shaped embeddings, and single-scale gradients}

We first setup a polymatroid embedding problem
which follows from the duality theorem of \cite{CKRV12}.
Fix a finite ground set $S$.
Given a function $\rho : \{0,1\}^S \to \{0,1\}$,
we define its {\em Lov\'asz extension $\hat \rho : [0,\infty)^S \to [0,\infty)$} by
$$
\hat \rho(z) = \int_0^{\infty} \rho(z^{\theta}) d\theta\,,
$$
where $z^{\theta} \in \{0,1\}^S$ has $(z^{\theta})_i=1$
whenever $z_i \geq \theta$.
Observe that for a constant $\alpha > 0$, we have $\hat \rho(\alpha \cdot z) = \alpha \cdot \hat \rho(z)$.
We will associate $2^S$ and $\{0,1\}^S$ via the mapping which sends
a subset $A \subseteq S$ to its characteristic function $\1_A \in \{0,1\}^S$.
Likewise, we will associate functions $S \to [0,\infty)$ with
elements of $[0,\infty)^S$.

In the rest of this section, we will consider families of functions $\mathcal F = \{\ell_v : E(v) \to [0,\infty)\}_{v \in V}$
associated to a graph $G=(V,E)$.  Given a length function $\len : E \to [0,\infty)$, we say that {\em $\mathcal F$ is adapted to $\len$}
if for every edge $e = \{u,v\} \in E$, we have
$$
\len(e) \leq \ell_u(e) + \ell_v(e)\,.
$$

\begin{theorem}[Duality Theorem, {\cite[Sec.~3]{CKRV12}}]
\label{thm:duality}
For any graph $G=(V,E)$ the following holds.
For any polymatroid capacities $\vec \rho = \{\rho_v : v \in V\}$
and any demands $\dem : V \times V \to [0,\infty)$,
\begin{equation} \label{eq:duality}
\mcf_G(\vec \rho,\dem) = \min_{\len, \{\ell_v\}} \left[ \frac{\sum_{v \in V} \hat \rho_v(\ell_v)}{\sum_{u,v \in V} \dem(u,v) d_{\len}(u,v)} \right],
\end{equation}
where the minimum is over all length functions $\len : E \to [0,\infty)$ on $G$
and all families $\{\ell_v : E(v) \to [0,\infty)\}_{v \in V}$ adapted to $\len$.
\end{theorem}

The preceding theorem shows that to prove flow/cut gaps,
it suffices to find for every given length function $\len$ and any $\{\ell_v\}_{v\in V}$ adapted to $\len$,
a set $S \subseteq E$
for which
$$
\Phi_G(S; \vec \rho,\dem) \leq C \cdot \frac{\sum_{v \in V} \hat \rho_v(\ell_v)}{\sum_{u,v \in V} \dem(u,v) d_{\len}(u,v)}
$$
for some constant $C > 0$.  This gives rise to an embedding problem with differs
from the classical one in a way which we now describe informally.

In the case of edge-capacitated flows and $L_1$ embeddings,
to satisfy the Lipschitz property, it suffices
to consider the stretch of each edge separately.  For vertex-capacitated flows,
and more generally polymatroid networks, we must {\em coordinate} the stretch
of the edges adjacent to a vertex.  In essence,
a vertex has to ``pay'' in the corresponding ``Lipschitz constant''
if {\em any} of its adjacent edges is stretched.  Thus we should
try as much as possible to stretch the edges adjacent to
a vertex simultaneously.

This makes some standard techniques (e.g. random embeddings into trees as in
\cite{GNRS04}) inappropriate for our study (although some
of the principles in \cite{GNRS04} will prove invaluable).
Certainly $\ell_1^{\mathrm{dom}}$ embeddings achieve this coordination
because they are (by definition) Lipschitz in every coordinate,
but as we mentioned earlier, they are insufficient for proving our main theorems.

To satisfy this goal, we must pay careful attention to the image of the edges
in our embeddings.  On the other hand, to overcome the limitations
of $\ell_1^{\mathrm{dom}}$, we will increase our target spaces
to include general metric trees.

\medskip
\noindent
{\bf Star-shaped mappings.}
Say that a graph $H$ is {\em star-shaped} if $H$ is the subdivision of some star graph.
Suppose that $G=(V,E)$ is a graph, $T$ is a tree,
and $\lambda : V \to V(T)$ is an arbitrary map.
For every $u,v \in V(T)$, let $P_{uv} \subseteq V(T)$ be the unique simple path
between $u$ and $v$ in $T$.  We say that $\lambda$ is a {\em star-shaped mapping}
if, for every $u \in V(T)$, the induced graph on $$\{ P_{uv} : v \in V(T), E(\lambda^{-1}(u),\lambda^{-1}(v))\neq \emptyset\}$$
is star-shaped.  In other words, if we consider the paths in $T$ which correspond to edges in $G$,
then all such paths emanating from the same vertex in $T$ should form a star-shaped subgraph.

In addition to controlling the {\em shape} of a mapping, we need to control
the lengths of the ``arms'' of the star simultaneously.  Fortunately (and this property
will be crucial to the approach of Section \ref{sec:retract}), we will only
need to bound the stretch over single scales.

\medskip
\noindent
{\bf Single-scale $\ell_{\infty}$ gradients.}
If we are given a metric graph $G=(V,E,\len)$ and a mapping $f : V \to (X,d_X)$ into a metric space $(X,d_X)$,
we make the following definition:  For any $\tau > 0$,
$$
|\nabla_{\tau} f(u)|_{\infty} = \sup \left\{ \frac{d_X(f(u),f(v))}{\len(u,v)} : \{u,v\} \in E \textrm{ and } \len(u,v) \in [\tau, 2\tau] \right\}\,.
$$

In Section \ref{sec:rounding}, we prove the following theorem
which shows how such mappings can be used for polymatroid flow/cut gaps.

\begin{theorem}[Main rounding theorem]
\label{thm:rounding}
Let $G=(V,E,\len)$ be a metric graph and suppose there exists a random metric tree $T$
and a random star-shaped mapping $F : V \to V(T)$ such that for some $K \geq 1$,
\begin{equation} \label{eq:rounding-max-scale-gradient}
\max_{v \in V} \sup_{\tau > 0} \mathbb E\,|\nabla_{\tau} F(v)|_{\infty} \leq K\,.
\end{equation}

Then for any family of functions $\left\{\ell_v : E(v) \to [0,\infty)\right\}_{v \in V}$
adapted to $\len$, and for any polymatroid capacities $\vec \rho = \{\rho_v\}_{v \in V}$
and demands $\dem : V \times V \to [0,\infty)$,
we have
\begin{equation}\label{eq:finalbnd}
\Phi_G(\vec \rho,\dem) \leq \frac{64 K \sum_{v \in V} \hat \rho_v(\ell_v)}{\sum_{u,v \in V} \dem(u,v) \cdot \E\,[d_T(F(u),F(v))]}\,.
\end{equation}
\end{theorem}

\subsection{The embedding theorem}

In light of
Theorem \ref{thm:rounding}, we are able to prove Theorem \ref{thm:polyoks}
by constructing appropriate random embeddings into trees.
In the present section we state our main embedding theorem
and give an outline of its proof.

\begin{theorem}\label{thm:main}
There exist constants $K,L \geq 1$ such that the following holds.
If $G=(V,E)$ is a metric planar graph, and $F \subseteq V$ is any face of $G$,
then there exists a random tree $T$ and random star-shaped mapping $\Lambda : V \to V(T)$
such that the following conditions hold.
\begin{enumerate}
\item For every $u \in V$ and $\tau > 0$, we have
\begin{math}
\E \,|\nabla_{\tau} \Lambda (u)|_{\infty} \leq K.
\end{math}
\item For every $u,v \in F$,
\begin{equation} \label{eq:main-colipschitz}
\E\left[ d_T(\Lambda(u),\Lambda(v)) \right]\geq \frac{d_G(u,v)}{L}\,.
\end{equation}
\end{enumerate}
\end{theorem}

Combined with the rounding theorem (Theorem \ref{thm:rounding}) and duality (Theorem \ref{thm:duality}),
this immediately yields Theorem \ref{thm:polyoks} and, in particular, a  vertex-capacitated
Okamura-Seymour theorem (Theorem \ref{thm:nodeoks}).

\begin{proof}[Proof of Theorem \ref{thm:polyoks}]
Fix a planar graph $G=(V,E)$, a face $F\subseteq V$ of $G$, demands $\dem:V\times V\to [0,\infty)$ supported
on $F$,
and polymatroid capacities $\vec{\rho}$.
By Theorem~\ref{thm:rounding}, there exists a length function $\len:E\to [0,\infty)$ and a family
$\{\ell_v:E(v) \to [0,\infty)\}_{v\in V}$ adapted to $\len$ such that
\begin{equation*}
\mcf_G(\vec \rho,\dem) = \frac{\sum_{v \in V} \hat \rho_v(\ell_v)}{\sum_{u,v \in V} \dem(u,v) d_{\len}(u,v)}\,.
\end{equation*}
Consider the metric planar graph $G=(V,E,\len)$.
By Theorem~\ref{thm:main} there exist a random tree $T$ and
a random star-shaped embedding $\Lambda:V\to T$ satisfying~\eqref{eq:rounding-max-scale-gradient}
with $K=1$, and~\eqref{eq:main-colipschitz} with some universal
constant $L>0$.
Applying Theorem~\ref{thm:rounding} with $\Lambda$, we conclude
\begin{equation*}
\Phi_G(\vec \rho,\dem) \stackrel{\eqref{eq:finalbnd} \land \eqref{eq:main-colipschitz}}{\le}
\frac{64KL \sum_{v\in V} \hat \rho(\ell_v)}{\sum_{u,v\in V} \dem(u,v) d_G(u,v)}
\stackrel{\eqref{eq:duality}}{=} 64KL \cdot \mcf_G(\vec{\rho},\dem). \qedhere
\end{equation*}
\end{proof}

We now give a brief outline of the proof of Theorem \ref{thm:main}.

\medskip
\noindent
{\bf First step:  Outerplanar graphs into random trees.}
Theorem \ref{thm:main} is proved in two main steps.  First, in Section \ref{sec:outerplanar},
we prove it for the special case of outerplanar graphs; this is precisely the situation
where the face $F$ satisfies $F=V$ in Theorem \ref{thm:main}.
It is known that outerplanar graph metrics embed into distributions
over dominating trees \cite{GNRS04}, but this is not sufficient
for our purposes; these maps are not star-shaped and do not satisfy the gradient
conditions.  Instead our proof is inspired by the result
of Charikar and Sahai \cite{CS01} stating that every outerplanar graph
metric can be embedded into the product of two trees with $O(1)$ distortion.  In particular,
each of these two embeddings must be $O(1)$-Lipschitz, so one hopes
that the star-shaped and gradient properties might be
achievable with their techniques.

Indeed, by following their basic induction
and using a heavily modified variant of their embedding,
we are able to obtain the desired result.
Unfortunately, for this purpose we are not able to obtain
a product of two trees; instead we need an entire distribution,
but this suffices in light of Theorem \ref{thm:rounding}.

\medskip
\noindent
{\bf Second step:  Retracting onto a face.}
The second step follows the approach of \cite{EGKRTT10} for proving
that face metrics (i.e. those metrics arising from taking the shortest-path
metric on a planar graph restricted to a face) embed into distributions
over dominating trees; this result was originally proved in \cite{LS09}
via a different method.
In \cite{EGKRTT10}, the authors randomly retract a planar graph $G=(V,E)$ onto a prescribed face $F \subseteq V$
in such a way that edges are not stretched too much in expectation.

Their embedding has the rather convenient property (not shared by previous
random retractions) that stars are mapped to stars,
satisfying our star-shaped ambitions.  Thus we are left to wrestle
with the $\ell_{\infty}$ gradient issue.  By using stronger properties
of known random partitioning schemes for planar graphs \cite{KPR93}---specifically
the fact that such partitions are ``padded'' in the language of \cite{GKL03,KLMN04}---we
are able to show that all single-scale $\ell_{\infty}$ gradients are $O(1)$
in expectation under the random retraction.
We remark that this mapping does {\em not} preserve global $\ell_{\infty}$ gradients
in expectation, and this is the main reason we have introduced the single-scale definition.
This pushes some non-trivial work to the rounding theorem in Section \ref{sec:rounding}
which must now show that all the scales can be rounded simultaneously.

\subsection{Preliminaries}

Here we review some additional definitions before
diving into the proofs.
We deal exclusively with
finite graphs $G = (V,E)$ which are free of loops
and parallel edges.
We will also write $V(G)$ and $E(G)$ for
the vertex and edge sets of $G$, respectively.
A {\em metric graph} is a graph $G$
equipped with a non-negative length function on edges $\len : E \to [0,\infty)$.
We will denote the pseudometric
space associated with a metric graph $G$ as $(V,d_G)$, where $d_G$ is the
shortest path metric according to the edge lengths.
Note that $d_G(x,y)=0$ may occur even when $x \neq y$, and
if $G$ is disconnected, there will be pairs $x,y \in V$ with
$d_G(x,y)=\infty$.  We allow both possibilities throughout the paper.
An important point is that {\em all length functions in the paper
are assumed to be reduced,} i.e. they satisfy
the property that for every $e = (u,v) \in E$, $\len(e) = d_G(u,v)$.
For $v \in V$ and $R \geq 0$, we write $B_G(v,R) = \{ u \in V : d_G(u,v) \leq R \}$.

In the present paper, paths in graphs are always simple, i.e., no vertex appears twice.
%\mnote{"path are simple" currently verified only for Sections 1,2.}
Given a metric graph $G$, we extend the length function to
paths $P \subseteq E$ by setting $\len(P) = \sum_{e \in P} \len(e)$.
We recall that for a subset $S \subseteq V$, $G[S]$
represents the induced graph on $S$.  For a pair of subsets $S,T \subseteq V$,
we use the notations $E(S,T) = \{ (u,v) \in E : u \in S, v \in T \}$
and $E(S) = E(S,S)$, and if $v \in V$, we write $E(v) = E(\{v\}, V\setminus\{v\})$.

Given a set $X$, a \emph{random map $F:X \to Y$} is shorthand for some probability space
$(\Omega,\mu)$ and a distribution over mappings
$\{F_\omega:X \to Y_\omega\}_{\omega\in \Omega}$.
Note that both $F$ and $Y$ are random variables.
In all our constructions, $X$ and $Y_{\omega}$ are finite sets.
When no confusion arises, probabilistic expressions containing $F$ and $Y$ should be understood as been taken over the probability space $(\Omega,\mu)$.
When we refer to a property of $Y$ or $F$, it should be understood that this property holds
\emph{for all}  $Y_\omega$ and $F_\omega : X \to Y_{\omega}$, $\omega\in \Omega$.

\section{Polymatroid networks and embeddings}
\label{sec:rounding}

Our primary goal in the present section is to prove
Theorem \ref{thm:rounding} which shows that random
tree embeddings can be used to bound flow/cut gaps in polymatroid networks.
We start in Section \ref{sec:ssrounding} by showing that
a fixed ``thin'' mapping into a tree can be use for rounding.
In Section \ref{sec:degreduce}, we prove the crucial property
that every star-shaped mapping into a tree can be converted
to a random thin map.  Finally in Section \ref{sec:rround},
we combine these results with a multi-scale analysis
to show that a suitable distribution over star-shaped mappings into random trees suffices
for rounding.

\subsection{Thin-star tree rounding}
\label{sec:ssrounding}

Consider a graph $G$, a connected tree $T$, and a map $f : V(G) \to V(T)$.
For every pair $u,v \in V$, let $P_{uv}$ denote the unique
simple path from $f(u)$ to $f(v)$ in $T$.
We say that $f$ is {\em $\Delta$-thin}
if, for every $u \in V(G)$, the induced graph
on $\bigcup_{v : \{u,v\} \in E(G)} P_{uv}$
can be covered by $\Delta$ simple paths in $T$
emanating from $f(u)$.
The next lemma gives a generalization of line-embedding rounding \cite{FHL08,CKRV12}
to arbitrary thin maps into trees.

\begin{lemma}\label{lem:deground}
Let $G=(V,E)$ be a graph, $T$ a connected metric tree,
and let $f : V \to V(T)$ be a $\Delta$-thin map.
Suppose that the set of functions $\left\{\ell_v : E(v) \to [0,\infty)\right\}_{v \in V}$ is such that
$d_T(f(u),f(v)) \leq \ell_u(e) + \ell_v(e)$ for every edge $e = \{u,v\} \in E$.

Then for any polymatroid capacities $\vec \rho = \{\rho_v\}_{v \in V}$
and demands $\dem : V \times V \to [0,\infty)$,
there exists a subset of edges $S \subseteq E$ such that
$$
\Phi_G(S;{\vec \rho,\dem}) \leq \frac{\Delta \sum_{v \in V} \hat \rho_v(\ell_v)}{\sum_{u,v \in V} \dem(u,v) \cdot d_T(f(u),f(v))}\,.
$$
\end{lemma}

\begin{proof}
For every edge $\{u,v\} \in E$, let $P_{uv}$ denote the unique simple
path between $f(u)$ and $f(v)$ in $T$.
For every $a \in E(T)$, we define the subset $S(a) \subseteq E$ by $$S(a) = \left\{\vphantom{\bigoplus} \{u,v\} \in E : a \in E(P_{uv}) \right\}\,.$$
Observe that if $a \in E(P_{xy})$ for some $x,y \in V$,
then $\sigma_{S(a)}(x,y)=1$.
Thus we have, for any $x,y \in V$,
\begin{equation}
\sum_{a \in E(T)} \len_{T}(a) \cdot \sigma_{S(a)}(x,y) \geq
d_T(f(x),f(y))\,.\label{eq:demands}
\end{equation}

Next, we give an upper bound on $\nu_{\vec \rho}(S(a))$ for every $a \in E(T)$.
First, arbitrarily orient the edges of $E(T)$.
Fix $a=(x,y) \in E(T)$ according to this orientation.
Consider any $\lambda \in [0,\len_T(a)]$.  For an edge $e \in S(a)$,
choose the orientation $e=(u,v)$ such that $P_{uv}$ traverses $a$ in the order $(x,y)$.
We will assign the edge $e$ to the vertex $u$ if
\begin{equation}\label{eq:assignrule}
d_T(f(u),x) + \lambda  \leq \ell_u(e)\,,
\end{equation}
and otherwise assign $e$ to the vertex $v$.
This gives, for every $\lambda \in [0,\len_T(a)]$, a valid assignment $g_{a,\lambda} : S(a) \to V$.
Integrating yields
\begin{equation}\label{eq:boundone}
\len_T(a) \cdot \nu_{\vec \rho}(S(a)) \leq \int_{0}^{\len_T(a)} \left(\sum_{v \in V} \rho_v(g_{a,\lambda}^{-1}(v))\right)\,d\lambda\,.
\end{equation}

Our next goal is to show that, for every $v \in V$, we have
\begin{equation}\label{eq:boundtwo}
\sum_{a \in E(T)} \int_{0}^{\len_T(a)} \rho_v(g_{a,\lambda}^{-1}(v))\,d\lambda \leq \Delta \hat \rho_v(\ell_v)\,.
\end{equation}
To this end, fix $v \in V$.  Since $f$ is $\Delta$-thin, there are
$k \leq \Delta$ paths $P_1, P_2, \ldots, P_k$ in $T$ emanating from $f(v) \in V(T)$
such that the following holds:  If $S(a)$ contains an edge with endpoint $v$,
then $a \in E(P_i)$ for some $i \in \{1,2,\ldots,k\}$.  Thus we can write
\begin{equation}\label{eq:paths}
\sum_{a \in E(T)} \int_{0}^{\len_T(a)} \rho_v(g_{a,\lambda}^{-1}(v))\,d\lambda
\leq \sum_{i=1}^{k} \sum_{a \in E(P_i)} \int_{0}^{\len_T(a)} \rho_v(g_{a,\lambda}^{-1}(v))\,d\lambda\,,
\end{equation}
and it suffices to bound each term of the latter sum separately.

To this end, fix $i \in \{1,2,\ldots,k\}$.
For $\theta \in [0,\len(P_i)]$, let
$$S_v(\theta) = \left\{\vphantom{\bigoplus} \{u,v\} \in E : f(u) \in V(P_i) \textrm{ and } \ell_v(\{u,v\}) \geq \theta \right\}\,.$$
By the assignment rule \eqref{eq:assignrule}, the fact that $d_T(f(u),f(v)) \leq \ell_u(\{u,v\}) + \ell_v(\{u,v\})$ for every $\{u,v\} \in E$,
and monotonicity of $\rho_v$,
we have
\begin{eqnarray*}
\sum_{a \in E(P_i)} \int_0^{\len_T(a)} \rho_v(g^{-1}_{a,\lambda}(v))\,d\lambda
&\leq& \int_0^{\infty} \rho_v(S_v(\theta))\,d\theta \\
&\leq & \int_0^{\infty} \rho_v(\ell_v^{\theta})\,d\theta \\
&=& \hat \rho_v(\ell_v)\,,
\end{eqnarray*}
where in the final line we have used the definition of the Lov\'asz extension $\hat \rho_v$ and
the notation: $\ell_v^{\theta}(\{u,v\}) = 1$ if $\ell_v(\{u,v\}) \geq \theta$ and $\ell_v^{\theta}(\{u,v\}) = 0$ otherwise.
Combining this with \eqref{eq:paths} yields \eqref{eq:boundtwo}.

Now interchanging sums and integrals in \eqref{eq:boundtwo} and summing \eqref{eq:boundone} over $a \in E(T)$ yields
$$
\sum_{a \in E(T)} \len_T(a) \cdot \nu_{\vec \rho}(S(a)) \leq \Delta \sum_{v \in V} \hat \rho_v(\ell_v)\,.
$$
Using this in conjunction with \eqref{eq:demands}, we have
\begin{eqnarray*}
\min_{S \subseteq E} \Phi_G(S;{\vec \rho,\dem}) &\leq & \min_{a \in E(T)} \frac{\nu_{\vec \rho}(S(a))}{\sum_{u,v \in V} \dem(u,v) \sigma_{S(a)}(u,v)} \\
&\leq&
\frac{\sum_{a \in E(T)} \len_T(a) \cdot \nu_{\vec \rho}(S(a))}{\sum_{a \in E(T)} \len_T(a) \sum_{u,v \in V} \dem(u,v) \sigma_{S(a)}(u,v)} \\
&\leq& \frac{\Delta \sum_{v \in V} \hat \rho_v(\ell_v)}{\sum_{u,v \in V} \dem(u,v) d_T(f(u),f(v))}\,,
\end{eqnarray*}
completing the proof.
\end{proof}

\subsection{Random thinning}
\label{sec:degreduce}

Next we show how an arbitrary star-shaped map into a tree can be converted into a random 4-thin map.

\begin{lemma}\label{lem:degreduce}
Let $G=(V,E)$ be a graph, $T$ a connected metric tree, and  let $f:V\to V(T)$ be a $1$-Lipschitz star-shaped map.  Then there exists a random
connected metric tree $T'$ and a random 4-thin map $F : V \to V(T')$ satisfying the following conditions:
\begin{enumerate}
\item $F$ is $1$-Lipschitz with probability one.
\item For every $u,v \in V$, we have
$$
\E\,d_{T'}(F(u),F(v)) \geq \frac12\, d_T(f(u),f(v))\,.
$$
\end{enumerate}
\end{lemma}

\begin{proof}
In what follows, for a tree $T$, we will use the notation $P_{xy}^T$ to denote the unique simple path between $x,y \in V(T)$.

We will proceed by constructing a random metric tree $T'$ and map $\Phi : V(T) \to V(T')$ and then showing that $F=\Phi\circ f$ satisfies the conclusion of the lemma.
Suppose that the tree $T$ is rooted at some fixed vertex.
For a vertex $x\in V(T)$, let $T_x$ be the subtree rooted at $x$. A \emph{vertical path} in a rooted tree is a  path $P$ in which every $u,v\in P$ have ancestor-descendant relationship.
%\mnote{Replacing ``root-leaf" with ``vertical" (those paths need not reach a leaf).}

\begin{claim}\label{claim:reduction}
For every $x\in V(T)$, there exists a random metric tree $T'_x$, and a random $1$-Lipschitz map
$\Phi_x : V(T_x) \to V(T'_x)$ which satisfies the following conditions:
 \begin{enumerate}
\item $\Phi_x$ maps every vertical path in $T_x$ isometrically to a vertical path in $T'_x$.
\item For all $v\in V(T_x)$, $\Phi_x|_{V(T_v)}=\Phi_v$.
\item For all $v\in V(T_x)$, the set of vertices $\{ \Phi_x(u) : u \in V(T_v), E(f^{-1}(u),f^{-1}(v)) \neq \emptyset \}$ can be covered by at most \emph{two} vertical paths emanating from $\Phi_x(v)$ in $T'_x$.
 \item For every $u,v \in V(T_x)$, we have
$$
\E\,[d_{T'}(\Phi_x(u),\Phi_x(v))] \geq \frac12\, d_T(u,v)\,.
$$
 \end{enumerate}
\end{claim}
\begin{proof}
We construct the map $\Phi_x$  by induction on the height of $x$ in $T$. When $x$ is a leaf, the statement is partically vacuous. The inductive step is carried in two steps: In the first step we construct a tree $\tilde T_x$ and a map $\tilde \Phi:V(T_x)\to V(\tilde T_x)$ as follows:
Let $u_1,\ldots, u_m$, be the children of $x$ in $T$, and let $T'_1, T'_2, \ldots, T'_m$ be the random trees, and $\Phi_1, \Phi_2, \ldots, \Phi_m$ be the randoms maps
resulting from applying the claim
inductively to each $T_{u_i}$. We construct the graph $\tilde T_x$ by replacing each $T_{u_i}$ with $T'_i$ in $T_x$. We put $\tilde \Phi(x)$ as the root of $\tilde T_x$, and for $v\in V(T_{u_i})$, we put $ \tilde \Phi(v)= \Phi_i(v)$. We also define $\tilde f=  \tilde \Phi\circ f$.
\medskip

Let  $S=\left\{P^{\tilde T_x}_{\tilde{\Phi}(x)\tilde v}: \tilde v\in \tilde T_x, {E(\tilde f^{-1}(x),\tilde f^{-1}(\tilde v))\neq \emptyset}\right\}$, and let $ H$ be the subgraph of $\tilde T_x$ induced by $S$. The map $f$ is a star-shaped, and  each map $\Phi_i$, maps root leaf paths in $T_{u_i}$ to root leaf paths in $ T'_i$, therefore the set $S$ is star-shaped.
Hence,  there exists $k\le m$ edge disjoint paths $P_1, P_2, \ldots, P_k$ in $\tilde T_x$ emanating from $\tilde \Phi(x)$ that cover $S$.

Let $\tilde T_1,\ldots, \tilde T_n$ be the connected components of $\tilde T_x$ after removing the edges of $H$, and let $\tilde v_j\in S$ be the root of the tree $\tilde T_j$.

We define the random tree $T'_x$, and the random  map $\Phi_x:T_x\to T'_x$ as follows.  Consider a root $r$ connected to two paths $B_1$ and $B_2$. We define $\Phi_x(x)=r$.
Moreover, we map each path $P_1, P_2, \ldots, P_{k}$ isometrically and independently  at random to one of the two paths $B_1$, and $B_2$. Then we complete the construction by gluing root of each tree $\tilde T_j$ to $\Phi_x(\tilde v_j)$ (the image of $\tilde v_j$ in either $B_1$ or $B_2$).
\medskip

It is  straightforward to check that $\Phi$ satisfies Claim~\ref{claim:reduction}(i).
Using the inductive hypothesis it is sufficient to check Claim~\ref{claim:reduction}(ii) with respect to the children of  $x$ in $T$, and for them the claim is easily seen to be true.

To verify Claim~\ref{claim:reduction}(iii), first note that for $v\in T_x\setminus \{x\}$ this condition holds by our inductive construction, and Claim~\ref{claim:reduction}(ii).
Moreover, For the case that $v=x$, all the vertices $u\in V(T_x)$ such that $E(f^{-1}(x),f^{-1}(u))\neq \emptyset$ are mapped to paths $B_1$ and $B_2$, therefore Claim~\ref{claim:reduction}(iii) holds for all $v\in V(T_x)$.

To verify Claim~\ref{claim:reduction}(iv), first note that for $u,v\in V(T_x)$, if $u,v\in V(T_{u_i})$ for some $i$, then $d_{T'_i}(\Phi_{u_i}(u),\Phi_{u_i}(v))=d_{T'_x}(\Phi_x(u),\Phi_x(v))$ and by our inductive construction $$\frac 12 d_T(u,v)\leq d_{T_x'}(\Phi_x(u),\Phi_x(v))\,.$$
 Moreover if $u$ and $v$ do not belong to the same subtree rooted at one of $x$'s children, then
with probability $1/2$, $d_T(u,v)=d_{T'_x}(\Phi_x(u),\Phi_x(v))$. Therefore Claim~\ref{claim:reduction}(iv) holds for all $u,v\in V(T_x)$, completing the proof of the claim.
\end{proof}

The map $f$ is star-shaped, therefore there must exist $k$ edge disjoint paths $P_1, P_2, \ldots, P_k$ emanating from $v$ that cover $\{f(u):(u,v)\in E(G)\}$ in $T$. Moreover, since these paths are edge disjoint,
at most one of these paths  is not contained in $T_{f(v)}$. Without loss of generality assume that $P_1, P_2, \ldots, P_{k-1}$ are contained in $T_{f(v)}$.
Let $T'$ and $\Phi$ to be the tree and the map resulting from applying claim to the root of $T$.
By our construction we can cover the image of the paths $P_1, P_2, \ldots, P_{k-1}$ by at most two paths emanating from $\Phi(f(v))$ in $T'$. Finally, every path in a rooted tree is a union of at most two vertical paths and each vertical path in $T$ is mapped to a vertical path in $T'$, the image of the path $P_k$ can be covered by at most two paths emanating from $\Phi(f(v))$ in $T'$, completing the proof of Lemma~\ref{lem:degreduce}.
\end{proof}

The next result follows from Lemma \ref{lem:degreduce} and Lemma \ref{lem:deground}.

\begin{corollary}\label{cor:starshaped}
Let $G=(V,E)$ be a graph, $T$ a connected metric tree,
and $f : V \to V(T)$ a star-shaped mapping.
Suppose that the set of functions $\left\{\ell_v : E(v) \to [0,\infty)\right\}_{v \in V}$ is such that
$d_T(f(u),f(v)) \leq \ell_u(e) + \ell_v(e)$ for every edge $e = \{u,v\} \in E$.

Then for any polymatroid capacities $\vec \rho = \{\rho_v\}_{v \in V}$
and demands $\dem : V \times V \to [0,\infty)$,
there exists a subset of edges $S \subseteq E$ such that
$$
\Phi_G(S;{\vec \rho,\dem}) \leq \frac{8 \sum_{v \in V} \hat \rho_v(\ell_v)}{\sum_{u,v \in V} \dem(u,v) \cdot d_T(f(u),f(v))}\,.
$$
\end{corollary}

\subsection{Rounding random star-shaped embeddings}
\label{sec:rround}

Finally, we are ready prove the main result of this section connecting embeddings to polymatroid flow/cut gaps.
We restate Theorem \ref{thm:rounding} here for the sake of the reader

\begin{theorem}
Let $G=(V,E,\len)$ be a metric graph and suppose there exists a random connected metric tree $T$
and a random star-shaped mapping $F : V \to V(T)$ such that for some $K \geq 1$,
$$\max_{v \in V} \sup_{\tau > 0} \mathbb E\,|\nabla_{\tau} F(v)|_{\infty} \leq K\,.$$
Then for any set of functions $\left\{\ell_v : E(v) \to [0,\infty)\right\}_{v \in V}$
that is adapted to $\len$, and
for any polymatroid capacities $\vec \rho = \{\rho_v\}_{v \in V}$
and demands $\dem : V \times V \to [0,\infty)$,
there exists a subset of edges $S \subseteq E$ such that
\begin{equation*}  %\label{eq:finalbnd}
\Phi_{\vec \rho,\dem}(S) \leq \frac{64 K \sum_{v \in V} \hat \rho_v(\ell_v)}{\sum_{u,v \in V} \dem(u,v) \cdot \E\,[d_T(F(u),F(v))]}\,.
\end{equation*}
\end{theorem}

\begin{proof}

Using the fact that $\hat \rho_v$ is monotone, we may first scale $\{\ell_v(e): v\in V\,,e\in E(v)\}$ down
and assume
that for $\{u,v\} \in E$, we have
\( %\begin{equation}\label{eq:true-equality}
\len(\{u,v\}) =\ell_u(\{u,v\}) + \ell_v(\{u,v\}).
\) %\end{equation}
Next, by rounding all the length functions up, we may assume that $\{\ell_v(e): v\in V\,,e\in E(v)\}$ are dyadic:
\begin{equation}\label{eq:dyadic}
\{\ell_v(e) : v \in V, e \in E(v) \} \subseteq \{ 2^h : h \in \mathbb Z \},
\end{equation}
and that for $\{u,v\} \in E$, we have
\begin{equation}\label{eq:equality}
\len(\{u,v\}) \geq \frac12 \left(\vphantom{\bigoplus}\ell_u(\{u,v\}) + \ell_v(\{u,v\})\right)\,.
\end{equation}

Now define the random functions $\{\tilde \ell_v : E(v) \to [0,\infty)\}_{v \in V}$ by
$$
\tilde \ell_v(\{u,v\}) =
\begin{cases}
0 & \textrm{if } \ell_v(\{u,v\}) < \ell_u(\{u,v\}) \\
2 \ell_v(\{u,v\}) \cdot \frac{d_T(F(u),F(v))}{\len(u,v)} & \textrm{otherwise.}
\end{cases}
$$
Then, by definition, we have $d_T(F(u),F(v)) \leq \tilde \ell_u(\{u,v\}) + \tilde \ell_v(\{u,v\})$
for every $\{u,v\} \in E$ since $\{\ell_v\}$ is adapted to $\len$.

\medskip

We define a new family $\{\hat \ell_v\}$ by $\hat \ell_v(e) = \sup \{ \tilde \ell_v(e') : \ell_v(e') \leq \ell_v(e) \}$.
Observe that $\hat \ell_v \geq \tilde \ell_v$ pointwise, thus by monotonicity, $\hat \rho_v(\hat \ell_v) \geq \hat \rho_v(\tilde \ell_v)$.
Additionally, we have $\ell_v(e) \leq \ell_v(e')$ if and only if $\hat \ell_v(e) \leq \hat \ell_v(e')$.
Thus the collections of edge sets $\{ \ell_v^{\theta} : \theta \in [0,\infty) \}$ and $\{ \hat \ell_v^{\theta} : \theta \in [0,\infty) \}$
are identical.

Enumerate the set of values $\{\ell_v(e) : e\in E(v)\} \cup \{0\}$ by $0 = \tau_0 <\tau_1 < \tau_2 < \cdots < \tau_k$
so that
\begin{equation}\label{eq:dyadic1}
\hat \rho_v(\ell_v) = \sum_{i=0}^{k-1} (\tau_{i+1}-\tau_i) \rho_v(\ell_v^{\tau_i}) \geq \frac12 \sum_{i=0}^{k-1} \tau_{i+1} \rho_v(\ell_v^{\tau_i})\,,
\end{equation}
where the latter inequality holds since $\tau_{i+1} \geq 2 \tau_i$ by \eqref{eq:dyadic}.

For $i=1,2,\ldots,k$, we can likewise set $\hat \tau_i = \max \{ \hat \ell_v(e) : \ell_v(e)=\tau_i \}$.
By construction, we have $0 = \hat \tau_0 \leq \hat \tau_1 \leq \hat \tau_2 \leq \cdots \leq \hat \tau_k$, and
$$
\hat \rho_v(\hat \ell_v) = \sum_{i=0}^{k-1} (\hat \tau_{i+1} - \hat \tau_i) \rho_v(\ell_v^{\tau_i})\,.
$$
Finally, define $\tilde \tau_i = \max \{ \tilde \ell_v(e) : \ell_v(e)=\tau_i \}$.
Observe that if $\tilde \tau_{i+1} \neq \hat \tau_{i+1}$, then $\hat \tau_{i+1}=\hat \tau_i$,
thus we can write
\begin{equation}\label{eq:dyadic2}
\hat \rho_v(\hat \ell_v)  \leq \sum_{i=0}^{k-1} \tilde \tau_{i+1} \rho_v(\ell_v^{\tau_i})\,.
\end{equation}

Using the definition of $\tilde \ell_v$, we have
$$
\tilde \tau_i = \max \left(\{0\} \cup \left\{ \tilde \ell_v(e) : \ell_v(e)=\tau_i \textrm{ and } \ell_v(e) \geq \frac12 \len(e) \right\}\right)\,.
$$

Furthermore, by \eqref{eq:equality}, if $\ell_v(e)=\tau_i$, then $\len(e) \geq \frac12 \tau_i$.  Thus,
$$
\mathbb E[\tilde \tau_i] \leq \left(\E |\nabla_{\tau_i/2} F(v)|_{\infty} + \E|\nabla_{\tau_i} F(v)|_{\infty}\right) \tau_i \leq 2K \tau_i\,.
$$
Using \eqref{eq:dyadic2} and \eqref{eq:dyadic1}, this implies
$$
\mathbb E[\hat \rho_v(\hat \ell_v)] \leq \sum_{i=0}^{k-1} \E[\tilde \tau_{i+1}] \rho_v(\ell_v^{\tau_{i}})
\leq 2K \sum_{i=0}^{k-1} \tau_{i+1}  \rho_v(\ell_v^{\tau_{i}}) \leq 4K \hat \rho_v(\ell_v)\,.
$$
Applying Corollary \ref{cor:starshaped} completes the proof.
\end{proof}

\section{Star-shaped embeddings of outerplanar graphs into trees}
\renewcommand{\ext}{\mathsf{glue}}
\label{sec:outerplanar}

Our goal is now to prove that every metric outerplanar graph admits a random Lipschitz, star-shaped
embedding into a random tree.

\begin{theorem}\label{thm:tree-embed}\label{thm:starpreserve}
There is a constant $K \geq 1$ such that the following holds.
Let $G=(V,E)$ be a metric outerplanar graph.  Then there is a random metric tree $T$ and
a random 1-Lipschitz, star-shaped mapping $F : V \to V(T)$ such that for every $u,v \in V$,
$\E[d_T(F(u),F(v))] \geq d_G(u,v)/K$.
\end{theorem}

We begin by setting up the notations and definitions needed to prove Theorem~\ref{thm:starpreserve}.

\subsection {Notation and definitions}

For a graph $G=(V,E)$, and $v\in V$, we use the notation $N_G(v)=\{u:(u,v)\in E\}$ to denote the set neighbors of $v$ in the graph $G$.
For a path $P$, we define the cycle $C(P,\ell)$ as the cycle obtained by connecting the endpoints of $P$ with an edge of length $\ell$. The length of the cycle $C$, is given by $\len(C)=\len(P)+\ell$. In this section, it is  helpful to think of cycles as continuous cycles and $V(P)\subseteq C$ as points on the cycle.

For a cycle $C$ and a point $p$ on the cycle we define $\mathsf{flat}(C,p)$ to be the path where $p$ is one end point and $x\in C$ is mapped to the point at distance $d_C(p,x)$ from $p$ on the path. Moreover for points $x,y\in C$ we use $d_{\mathsf{flat}(C,p)}(x,y)=|d_C(x,p)-d_C(y,p)|$ to denote the distance between $x$ and $y$ on the path $\mathsf{flat}(C,p)$.
See Figure~\ref{fig:flat} for an example.
\begin{figure}[htbp]
\begin{center}
  \includegraphics[height=1.5in]{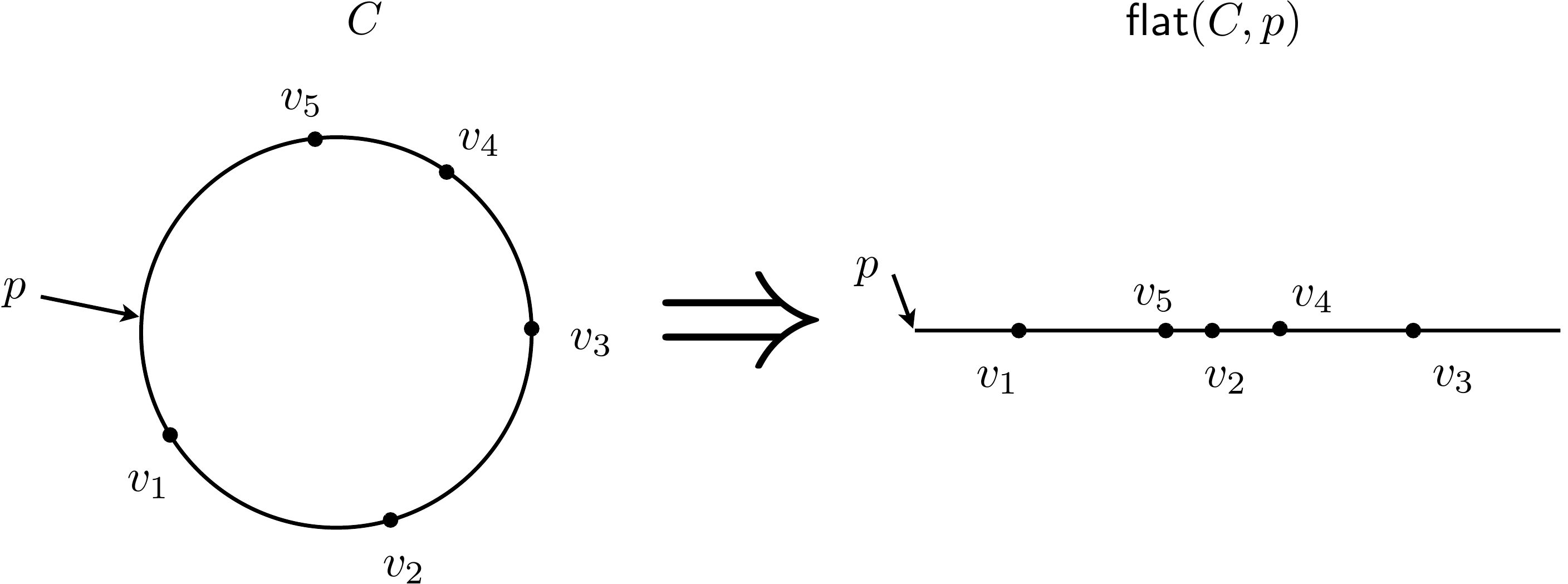}
\caption{The flattening of the cycle $C$.}
\label{fig:flat}
\end{center}
\end{figure}

For two paths $P=(u_1,\ldots,u_m)$ and $Q=(v_1,\ldots, v_n)$ with the same length, we define the  \emph{glue}ing of $P$ and $Q$ as follows. We first identifying  the end points $u_1$ with $v_1$, and $u_m$ with $v_n$ to specify the end points of the resulting path. Then we map each point $x\in V(P)\cup V(Q)$ so that the distance between $x$ and the end points of the path is preserved.

\begin{figure}[htbp]
\begin{center}
  \includegraphics[height=2in]{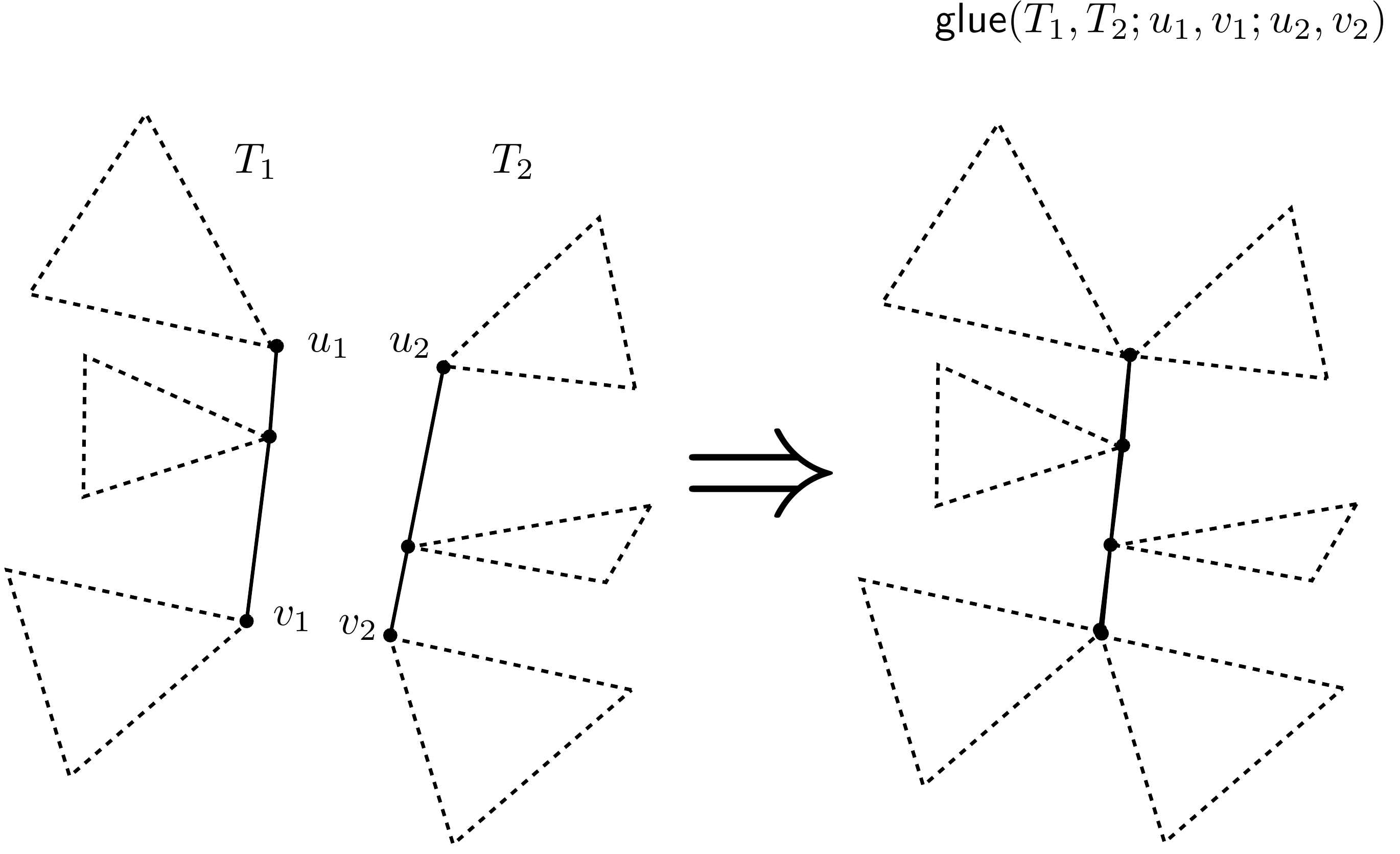}
\caption{Gluing of the trees $T_1$ and $T_2$.}
\label{fig:ext}
\end{center}
\end{figure}

Finally, for two given trees, $T_1$ and $T_2$, and pairs of vertices $u_1,v_1\in V(T_1)$ and $u_2,v_2\in V(T_2)$ such that $d_{T_1}(u_1,v_1)=d_{T_2}(u_2,v_2)$, we define $\ext(T_1,T_2;u_1,v_1;u_2,v_2)$ as the tree resulting from gluing  the trees $T_1$ and $T_2$ on the unique path between $u_1$ and $v_1$ in $T_1$ and $u_2$ and $v_2$ in $T_2$. See Figure~\ref{fig:ext}.

\subsection{Framework}\label{subsection:alg}

    Our approach to Theorem~\ref{thm:starpreserve} employs the framework of Charikar and Sahai (see Theorem~4 in \cite{CS01}).
Any outerplanar graph can be constructed by
considering a sequence of paths $P_i$, and then doing the following:
Start with $G_1 = P_1$. At step $i$, we consider some
edge $e_i = (u_i, v_i)$ on the outer face of $G_i$, and obtain $G_{i+1}$
by either attaching the endpoints of $P_i$ to $u_i$ and $v_i$, or by attaching only one endpoint of $P_i$ to
either $u_i$ or $v_i$.

In this section we only consider biconnected outerplanar graphs (so the endpoints of $P_i$ are always attached to $u_i$ and $v_i$), since we can simply take the embedding of biconnected components of a graph that are connected by a single vertex into trees, and glue the trees on the image of the common vertex to obtain an embedding for the whole graph.

We also use the concept of a slack structure
\cite{GNRS04}. We say that an outerplanar graph has an $\alpha$-slack
structure if it can be built out of paths $P_i$ such that the length
of any path $P_i$ which attaches to both endpoints of an edge
$e_i$ is at least $\alpha$ times the length of $e_i$. The following lemma is a straightforward  generalization of a fact from \cite{GNRS04}, where it is proved
for $\alpha=2$.

\begin{lemma}[\cite{GNRS04}]  \label{lem:slack}
Consider any $\alpha \geq1$.  Given an outerplanar metric  graph $G =
(V,E,\len_G)$, there is an outerplanar metric graph
$H = (V,E',\len_H)$ with $E'\subseteq E$, and such that $H$ has an
$\alpha$-slack structure.  Furthermore, $d_G \geq d_H \geq  (1/\alpha)d_G$,
and for every $(u,v)\in E'$,
\begin{equation} \label{eq:lem:slack}
d_H(u,v)=\len_H(u,v).
\end{equation}
\end{lemma}

Thus, by incurring distortion at most $\alpha$, we may
assume that
the outerplanar graph $G$ has an $\alpha$-slack
structure. We will build our embedding inductively based
on the sequence of the paths $P_1.\ldots, P_m$ provided by Lemma~\ref{lem:slack}.

\medskip
\noindent
{\bf Random extension.}
Given an embedding of a metric graph $G$ into a random metric tree $T$, $F:V(G)\to V(T)$
and a new path $P$ attached to the points $u,v\in V(G)$, we extend the embedding of $G$ to an embedding  for $\hat G=G\cup P$ into a random tree $\hat T$, using the following operation.
Let $C=C(P,d_T(F(u),F(v)))$. %It is helpful to think of $C$ as continuous cycle with vertices as points on the cycle.
To extend the tree $T$, we choose two \emph{anchor} points $p, q\in C$, and map the vertices of $C$ onto two paths $L=\mathsf{flat}(C,p)$ and $R=\mathsf{flat}(C,q)$.

We put $\hat T = \ext(T,L;F(u),F(v);u,v)$ with probability $1/2$ and
$\hat T= \ext(T,R;F(u),F(v);u,v)$ with
probability $1/2$.
This specifies a random mapping $\hat F : V(G) \to V(\hat T)$.
Since it will be clear from context which vertices we are gluing onto, we will use the notations $\ext(T,L)$ and $\ext(T,R)$ without specifying the vertices.
Note that the gluing can be
done if and only if $d_L(u,v)=d_R(u,v)=d_C(u,v)=d_T(F(u),F(v))$.
Moreover,
If the map $F : V(G) \to V(T)$ is $1$-Lipschitz, then so is the extension $\hat F$.

A significant difference between our construction and that of \cite{CS01} is in the way we choose the anchor points.
For our purposes, it is not enough to simply look at the $\alpha$-slack graph; we need to use the structure of the original graph when we choose the anchor points in order to maintain the star-shaped property.  The algorithm of \cite{CS01} is able to construct an embedding using only two trees, while we embed the graph into a distribution over trees.
In the next section, we prove a distortion bound for this embedding based on the distance between the anchor points in the cycle.% choose the anchor points to satisfy the conditions of Theorem~\ref{thm:starpreserve}.

\subsection{Bounding the distortion}

Before we can state the main lemma of this section, we need the following definition.
For a cycle $C$, and points $u,v\in C$ we say that a pair of points $p,q\in C$ is {\em $(\alpha,\beta)$-apart with respect to
another pair $\{u,v\}$} if $d_C(p,q)=\alpha\len(C)$ and for $a\in \{u,v\}$ and $b\in \{p,q\}$:
$$
\beta\, \len(C) \leq d_C(a,b)\leq\left({\frac 1 2-\beta}\right) \len(C).
$$

%In the rest of this section, if it is clear from the context which vertices are $u$ and $v$, we may use $(\alpha,\beta)$-apart without specifying the vertices $u$ and $v$.

We now state a lemma capturing our main inductive step.

\begin{lemma}\label{lem:distortion}
Let $G$ be a graph, $T$ be random metric tree, and let $F:V(G)\to V(T)$ be a random $1$-Lipschitz map such that $\E[d_T(F(x),F(y)] \ge d_G(x,y)/6$ for every $x,y\in V(G)$.
Let $\hat G$ be a graph constructed by attaching a path $P$ with
\begin{equation}\label{eq:lem:dist:slack}
\len(P)\geq 160\cdot d_G(u,v)
\end{equation}
 onto a pair of vertices $u,v\in V(G)$.
 Let $C=C(P,d_T(F(u),F(v)))$, and  $p,q \in C$
 be any pair of points that are $(1/6,1/16)$-apart with respect to $\{u,v\}$ in $C$, and let $\hat T$ be the random extension of $T$ by $C$ with respect to the anchor points $p$ and $q$.
The embedding $\hat F:V(\hat G)\to V(\hat T)$ is also 1-Lipschitz
and such that for all $x,y\in V(\hat G)$,
$\E[d_{\hat T}(\hat F(x),\hat F(y)] \ge d_{\hat G}(x,y)/6$.
\end{lemma}

\medskip
We will use mainly the following two properties of $(\alpha,\beta)$-apart pairs in the proof of Lemma~\ref{lem:distortion}.

\begin{observation}\label{obs:simple}
For any $\beta \in [0,1/2]$, the following holds.  Suppose $C$ is a cycle and $a,b \in C$ are such that $$\beta\, \len(C) \leq d_C(a,b)\leq\left({\frac 1 2-\beta}\right) \len(C).$$
 Then for any $x,y \in C$ with $\max\{d_C(x,a), d_C(y,a)\} \leq \beta \len(C)$, we have
$$
d_{\mathsf{flat}(C,b)}(x,y)=d_{C}(x,y).
$$
\end{observation}

\begin{lemma}\label{lem:flat}
Let $C$ be a cycle. For  $\alpha\in[0,1/4]$ and  $p,q\in C$ such that $d_C(p,q)=\alpha \len(C)$, the following holds. For any pair of vertices $x,y\in C$,
\begin{equation} \label{eq:lem:flat}
d_{\mathsf{flat}(C,p)}(x,y)+d_{\mathsf{flat}(C,q)}(x,y)\geq 4\alpha d_C(x,y).
\end{equation}
\end{lemma}
\begin{proof}
We divide the analysis into two cases.
If neither $p$ or its antipodal point $\bar p$ (the point at distance $\len(C)/2$ from $p$ on the cycle) lie on a shortest path between $x$ and $y$ then a simple application of Observation~\ref{obs:simple} implies that
$$
d_{\mathsf{flat}(C,p)}(x,y)+d_{\mathsf{flat}(C,q)}(x,y)\geq d_{\mathsf{flat}(C,p)}(x,y)= d_C(x,y)\geq 4\alpha d_C(x,y),
$$
and an analogous inequality holds if neither $q$ or it antipodal point $\bar q$
lie on a shortest path between $x$ and $y$.

Next, suppose that $p'\in\{p,\bar p\}$ and $q'\in \{q,\bar q\}$ lie on the same shortest path between $x$ and $y$.
 It is easy to check that ${\mathsf{flat}(C,p)}$ is isometric to ${\mathsf{flat}(C,p')}$
 and ${\mathsf{flat}(C,q)}$ is isometric to ${\mathsf{flat}(C,q')}$.
 We have $d_C(p',q')\in \{\alpha \len(C),(\frac 12 -\alpha)\len(C)\}$ and $\alpha\in[0,1/4]$, therefore $d_C(p',q')\geq \alpha\len(C)$.
Since both $p'$ and $q'$ are on the same shortest path between $x$ and $y$ we can write,
\begin{align*}
d_{\mathsf{flat}(C,p)}(x,y)+d_{\mathsf{flat}(C,q)}(x,y)&= d_{\mathsf{flat}(C,p')}(x,y)+d_{\mathsf{flat}(C,q')}(x,y)\\
&=|d_C(p',x)-d_C(p',y)|+|d_C(q',y)-d_C(q',x)|\\
&\geq|d_C(p',x)-d_C(q',x)+d_C(q',y)-d_C(p',y)|\\
&=|d_C(p',x)-d_C(q',x)|+|d_C(q',y)-d_C(p',y)|\\
&=2d_C(p',q')
\geq 2\alpha \len(C)
\geq 4\alpha d_C(x,y). \qedhere
\end{align*}
\end{proof}

%Now, we show that for suitable values of $\alpha$ and $\beta$, if we choose the anchor points such that they are $(\alpha,\beta)$-apart, then our embedding has constant distortion.

\medskip
\begin{proof}[Proof of Lemma~\ref{lem:distortion}]
Since $F$ is $1$-Lipschitz and the random extension preserves the 1-Lipschitz condition,  $\hat F$ is also $1$-Lipschitz.
%Let $T_L =\ext(T,\mathsf{flat}(C,p))$ and $T_R=\ext(T,\mathsf{flat}(C,q))$ be the two extensions of the tree $T$. For $x,y\in V(\hat G)$, we will use the notation $d_L$ and $d_R$ to denote the distance between images $x$ and $y$ in $T_L$ and $T_R$.
We divide the analysis of the expected contraction of the pairs $x,y\in V(\hat G)$ into three cases.

\medskip
\noindent
{\bf Case I. $x,y\in V(G)$:} In this case,
$$\E[d_{\hat T}(\hat F(x),\hat F(y))]=\E[d_{T}(F(x),F(y))]\geq  \frac 1 6d_{G}(x,y)=\frac 1 6d_{\hat G}(x,y).$$

\medskip
\noindent
{\bf Case II. $x\in V(P)$ and $y\in V(G)$:}
Observe that a shortest path in $\hat G$ connecting $x$ to $y$ must pass through either $u$ or $v$. Suppose, without loss of generality that
$d_{\hat G}(x,y)=d_{P}(x,u)+ d_G(u,y)$.
%Suppose that $d_P(x,u)\leq d_P(x,v)$, and
Let $w\in V(T)$  be the closest vertex to $F(y)$ (with respect to $d_T$)
from the unique path connecting $F(u)$ and $F(v)$ in $T$.
In this case by \eqref{eq:lem:dist:slack}, we have
\[ d_T(F(u),w)\leq d_T(F(u),F(v))\leq d_G(u,v) \leq \len(C)/160. \]
Suppose first that $d_{\hat G}(u,x)\leq \len(C)/16$. Since $p$ and $q$ are $(1/6,1/16)$-apart,  Observation~\ref{obs:simple} implies that
%the shortest path between $x$ and $w$ in $C$ is mapped isometrically into both $L$ and $R$. Therefore,
\begin{equation}\label{eq:lem:dist:I}
d_{\hat T}(\hat F(x),\hat F(y))=
d_{C}(x,u)+d_{T}(F(u),w)+d_{T}(w,F(y))%=d_{\hat G}(u,x)+d_{T}(u,w)+d_{T}(w,y)
=d_{\hat G}(x,u)+d_{T}(F(u),F(y)).
\end{equation}
Suppose next that $d_{\hat G}(u,x)> \len(C)/16$. Denote by $w'\in C$
the point with $d_C(u,w')=d_T(F(u),w)$.
Lemma~\ref{lem:flat} implies
\begin{align*}
\E[d_{\hat T}(\hat F(x),\hat F(y))\mid F]&= \E[d_{\hat T}(\hat F(x),w)\mid F]+d_T(w,F(y))\\
&\overset{\eqref{eq:lem:flat}}\geq \frac 1 3 d_{C}(x,w')+d_T(w,F(y))\\
&\geq \frac 1 3 (d_{C}(u,x)-d_{C}(u,w'))+(d_T(F(u),F(y))-d_T(F(u),w))\\
&\geq \frac 1 3 (d_{C}(u,x)-d_{C}(u,v))+(d_T(F(u),F(y))-d_T(F(u),F(v))).
\end{align*}
Since $d_{\hat G}(u,x)\leq d_{C}(u,x) + d_{\hat G}(u,v) - d_C(u,v)$, we have
\begin{align}
\E[d_{\hat T}(\hat F(x),\hat F(y))\mid F]&\geq \frac 1 3 (d_{\hat G}(u,x)-d_{\hat G}(u,v))+(d_T(F(u),F(y))-d_{\hat G}(u,v))\nonumber\\
&\geq \frac 1 3 d_{\hat G}(u,x)-\frac {4} 3d_{\hat G}(u,v)+d_T(F(u),F(y))\nonumber\\
&\overset{\eqref{eq:lem:dist:slack}}\geq \frac 1 3 d_{\hat G}(u,x)-\frac {16\cdot 4}{ 3\cdot 160}d_{\hat G}(u,x)+d_T(F(u),F(y))\nonumber\\
&\geq \frac 1 6 d_{\hat G}(u,x)+d_T(F(u),F(y)).\label{eq:lem:dist:II}
\end{align}
Putting \eqref{eq:lem:dist:I} and \eqref{eq:lem:dist:II} together, we can conclude that
\[
\E\left[d_{\hat T}(\hat F(x),\hat F(y))\right]\geq\frac 1 6 d_{\hat G}(u,x)+ \E \left[d_T(F(u),F(y))\right]  \geq
\frac 1 6(d_{\hat G}(x,u)+ d_{\hat G}(u,y)) =
\frac 1 6 d_{\hat G}(x,y).
\]

\medskip
\noindent
{\bf Case III. $x,y\in V(P)$:} In this case we divide the problem into three cases again.
If no shortest path between $x$ and $y$  in the graph $\hat G$ goes through the edge $(u,v)$, then
$$
d _C(x,y) \geq   \min\{d _{\hat G}(x,y), \len(P)-d_{\hat G}(x,y)\}.
$$
Moreover, $\eqref{eq:lem:dist:slack}$ implies that  $d_{\hat G}(x,y)\leq \frac 1 2(\len(P)+d_{\hat G}(u,v))\leq \frac {81}{160}\len(P)$. Thus,
$$
d _C(x,y) \geq \min \Bigl \{d _{\hat G}(x,y), \frac {160}{81}d_{\hat G}(x,y)-d_{\hat G}(x,y)\Bigr\} \geq \frac{79}{81}d_{\hat G}(x,y)\geq \frac 1 2d_{\hat G}(x,y).
$$
Hence, using Lemma~\ref{lem:flat} we can conclude that
\begin{align*}
\E[{d_{\hat T}(\hat F(x),\hat F(y))}\mid F]\geq  \frac 1 3d _C(x,y)\geq  \frac 1 6 d_{\hat G}(x,y).
\end{align*}

If a shortest path between $x$ and $y$ in the graph $\hat G$ passes through the edge $(u,v)$
and $d_{\hat G}(x,y)> \len(P)/16$, then by Lemma~\ref{lem:flat}
\begin{align*}
\E[{d_{\hat T}(\hat F(x),\hat F(y))}\mid F]\!\geq\! \frac 1 3 d_C(x,y)
\!\geq\!  \frac 1 3(d_{\hat G}(x,y)-d_{\hat G}(u,v))
\overset{\eqref{eq:lem:dist:slack}}> \frac 1 3(1-16/160)d_{\hat G}(x,y)
>  \frac 1 6 d_{\hat G}(x,y).
\end{align*}

Finally, we consider the case where a shortest path between $x$ and $y$  in the graph $\hat G$ passes through the edge $(u,v)$
and $d_{\hat G}(x,y)\le \len(P)/16$.  Suppose that $d_{\hat G}(u,x)\leq d_{\hat G}(v,x)$. It is easy to check that we also have $d_{C}(u,x)\leq d_{C}(v,x)$, and that the shortest path between $x$ and $y$ in $C$ also passes through the edge $(u,v)$.
Hence,
by Observation~\ref{obs:simple},
$$
d_{\hat T}(\hat F(x),\hat F(y))=d_C(x,y)=d_{\hat G}(x,u)+d_T(F(u),F(v))+d_{\hat G}(v,y).
$$
Thus,
\begin{align*}
\E\left[d_{\hat T}(\hat F(x),\hat F(y))\right]
&\geq \E \left[d_{\hat G}(x,u)+d_T(F(u),F(v))+d_{\hat G}(v,y)\right]\\
&=d_{\hat G}(x,u)+d_{\hat G}(y,v)+\E \left[d_T(F(u),F(v))\right]\\
&\geq d_{\hat G}(x,u)+d_{\hat G}(y,v)+\frac 1 6d_{\hat G}(u,v)\\
&\geq \frac 1 6d_{\hat G}(x,y),
\end{align*}
completing the proof.
\end{proof}

\subsection{The star-shaped property}

To prove Theorem~\ref{thm:starpreserve}, we need to choose the anchor points such that the resulting map is star-shaped.
The following lemma and its corollary provide the main tools necessary to achieve this.

\begin{lemma}\label{lem:choice}
For any $\alpha\in(0,1/6)$, $\beta\leq \frac 1 4-\frac 3 2\alpha$, and  $\delta < \alpha$ the following holds. Let $P$ be a path with endpoints $u,v$, and let $C=C(P,\delta \len(P))$. For any finite subset $S\subseteq C$, there are points $p,p'\in C$ that are $(\frac 1 4-\frac 3 2\alpha-\beta,\beta)$-apart with respect to $u$ and $v$, and moreover for $q\in \{p,p'\}$,
\begin{enumerate}[\hspace{0.2in}(a)]
\item $d_C(q,u)\leq d_C(q,v)$;
\item For  $x\in V(P)$ with $d_P(x,u)\leq \left(\frac 12+ \alpha\right) \len(P)$, we have $d_C(q,x)\leq d_C(q,u)$;
\item For any two distinct elements $x,y\in S$, we have $d_C(q,x)\neq d_C(q,y)$.
\end{enumerate}
\end{lemma}

\begin{proof}
For $\eta\in (\delta,\alpha)$, it is easy to verify that the two points $p, p' \in C$ with
\begin{eqnarray*}
d_C(u,p)&=&\left(\frac 1 4+\frac {3\alpha} 2-\eta\right)\len(C) \\
d_C(u,p')&=&\left(\frac 1 2-\eta -\beta\right)\len(C)
\end{eqnarray*}
 satisfy conditions (a) and (b). Condition (c) also holds for almost all $\eta \in (\delta,\alpha)$.
%The first condition implies that $q$ lies on the half cycle that is closer to $u$ than $v$.
%The second condition implies that the antipodal point of $q$ in the cycle is at least at distance $\beta \len(C)$ from the segment $(u,v)$, and the third condition implies that $d(u,q)\geq \len(C)/4+\alpha/2$.
%The intersection of these three intervals is an interval of length
%$$
%\len(C)/2 - \left(\frac 1 4 -\frac \alpha 2\right)\len(C) -\alpha (d_C(u,v)+\beta \len(C))\geq (\frac 14-\frac 3 2\alpha-\beta)\len(C).
%$$
%Therefore, we can choose $p$ and $p'$ from this interval of distance exactly $(\frac 1 4-\frac 3 2\alpha-\beta)\len(C)$.
\end{proof}

\begin{corollary}\label{cor:choice}
For $\delta\leq 1/160$, the following holds. Let $P$ be a path with endpoints $u,v$, and let $C=C(P,\delta \len(P))$. For any finite
subset $S\subseteq C$, there are points $p,p'\in C$ that are $(1/6,1/16)$-apart with respect to $u$ and $v$, and moreover for $q\in \{p,p'\}$,
\begin{enumerate}[\hspace{0.2in}(a)]
\item $d_C(q,u)\leq d_C(q,v)$;
\item For $x\in V(P)$ with $d_P(x,u)\leq \left(\frac 12+ \frac 1{160}\right) \len(P)$, we have $d_C(q,x)\leq d_C(q,u)$;
\item For any two distinct elements $x,y\in S$, we have $d_C(q,x)\neq d_C(q,y)$.
\end{enumerate}
\end{corollary}

The next lemma is the
final ingredient that we need to prove Theorem~\ref{thm:starpreserve}.
\begin{lemma}\label{lem:oneend}
Let $G=(V,E)$ be a biconnected outerplanar graph with outer face $C$, and let $d_C$ be the path pseudometric in the induced graph $G[C]$. Consider any point $p \in C$  and edge $(u,v)$ on the outer face, and suppose that $d_C(p,u)< d_C(p,v)$. Then one of the following two conditions hold:
\begin{itemize}
\item $\forall w\in N_G(v): d_C(p,w)\geq d_C(p,u)$;
\item $\forall w\in N_G(u): d_C(p,w)\leq d_C(p,v)$.
\end{itemize}
\end{lemma}

\begin{figure}[htbp]
\begin{center}
 \includegraphics[height=1.5in]{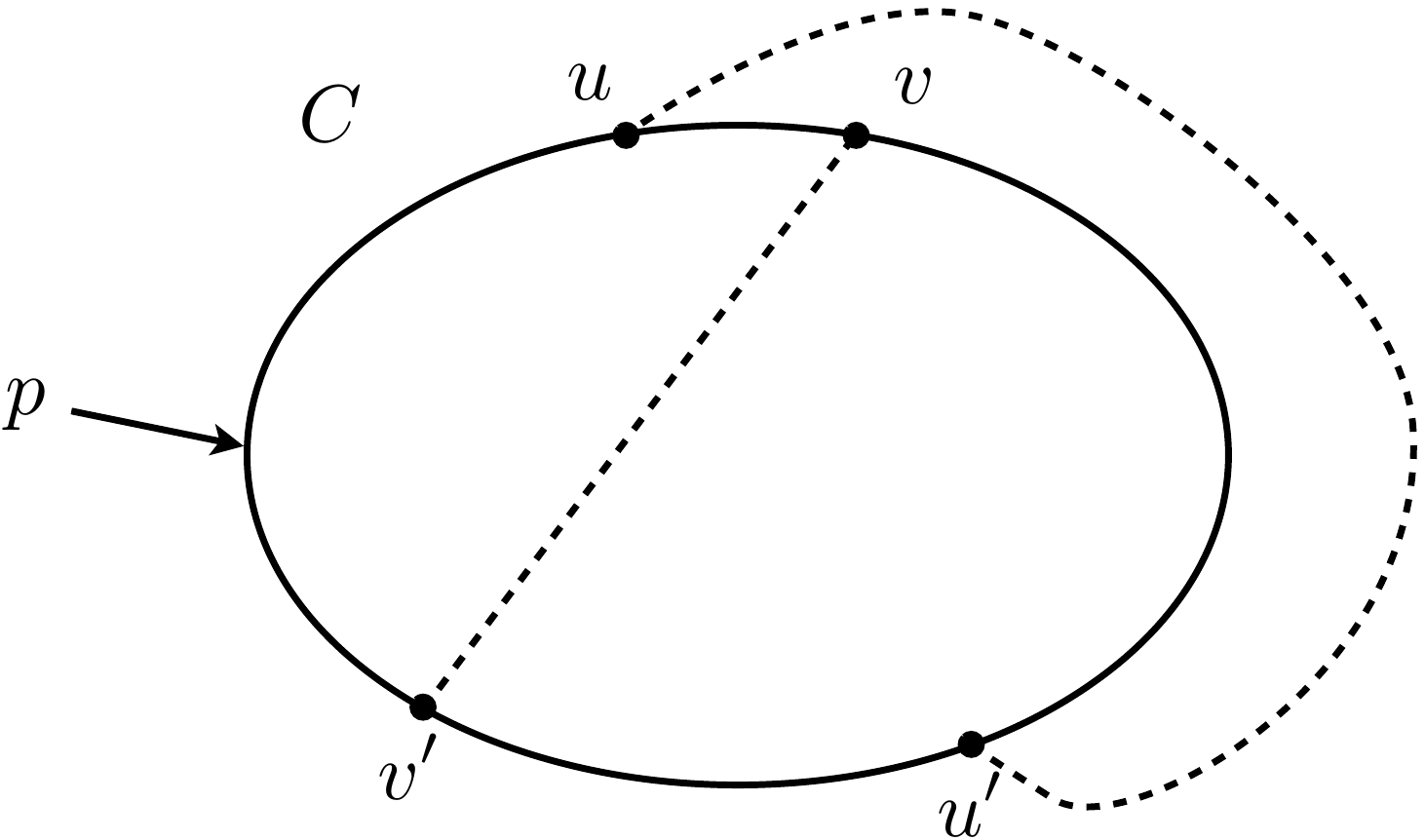}
\caption{Positions of $u, u', v$, and $v'$ on the cycle.}
\label{fig:k4}
\end{center}
\end{figure}

\begin{proof}
For the sake of contradiction suppose that none of the conditions of the lemma hold. Let $v'\in N_G(v)$, be such that $d_C(p,v')<d_C(p,u)$ and let $u'\in N_G(u)$ be such that $d_C(p,u')> d_C(p,v)$ (See Figure~\ref{fig:k4}.) It is easy to check that the union of the cycle $C$ with the edges  $(u,u')$ and $(v,v')$ contains $K_4$ as minor, therefore $G$ cannot be outerplanar, and we have arrived at a contradiction.
\end{proof}

We can now prove Theorem~\ref{thm:starpreserve}.

\medskip
\begin{proof}[Proof of Theorem~\ref{thm:starpreserve}]
Without loss of generality, we may assume that $H$ is 2-connected, as it is trivial to construct the desired
embedding for $H$ from embeddings for each of its 2-connected components.
To construct the embedding, we first use Lemma~\ref{lem:slack} to transform the graph into an $160$-slack graph $H$, with a decomposition into paths $P_1,\ldots P_m$.
Then, we use the random extension algorithm from Section~\ref{subsection:alg} to inductively build an embedding from $H$ into random trees.
To avoid ambiguity, in what follows for $x,y\in V(G)$ and and a mapping $F:V(G)\to V(T)$, we will use the notation $P_{F(x)F(y)}^T$ to denote the unique simple path between $F(x)$ and $F(y)$ in $T$.

We start with the graph $H_1=P_1$ and at step $i\geq 2$ we construct the graph $H_{i}$ by attaching the path $P_i$ to the edge $(u_i,v_i)\in E(H_{i-1})$.  (Note that since $H$ is $2$-connected, the endpoints of $P_i$ are always attached to both $u_i$ and $v_i$.)

We construct the embedding for $H_{i}$ from an embedding $F_{i-1}:V(H_{i-1})\to V(T_{i-1})$ as follows. We  use random extension with anchor points which are $(1/6,1/16)$-apart with respect to $u_i$ and $v_i$ to extend $T_{i-1}$ and $F_{i-1}$ to $T_{i}$ and $F_{i}:V(H_{i})\to V(T_{i})$. Moreover, we choose the anchor points so that the resulting map is injective and it  maintains the following additional property:

 or all $i \geq 1$ and any edge $(x,y)$ on the outer face of $H_i$,  there is at least one endpoint $x$ such that for all $w\in N_G(x)\cap V(H_i)$,
\begin{equation}\label{eq:def:good}
\left( P^{T_i}_{F_i(x)F_i(w)}\subseteq P^{T_i}_{F_i(x)F_i(y)} \right) \lor \left( P^{T_i}_{F_i(x)F_i(w)}\cap P^{T_i}_{F_i(x)F_i(y)}=\{F_i(x)\} \right).
\end{equation}
We call a vertex that satisfies the above property a \emph{good} vertex for the edge $(x,y)$ with respect to $V(H_i)$ and $F_i$.%\mohnote{Added $F_i$, since it also depends on the map}
% Note that here, as defined in Section~\ref{sec:preliminaries},
%$F_i:V(H_i) \rightarrow T_i$ stands as a shorthand for ``every mapping in the support of the random mapping $F_i:V(H_i)\to T_i$".

If the edge $(u_i,v_i)$, we choose the anchor points on the cycle $C=C(P_i,d_T(u_i,v_i))$ such that they satisfy conditions (a) and (b) of Corollary~\ref{cor:choice}, with $v$ being a good vertex of the edge $(u_i,v_i)$ with respect to $T$, and $\delta=d_T(u_i,v_i)/\len(P_i)$. See Figure~\ref{fig:flat2} for an example.
Note that Corollary~\ref{cor:choice}(c) implies that we can always find such anchor points while maintaining the invariant that the map is injective.

\begin{figure}[htbp]
\begin{center}
  \includegraphics[height=1.5in]{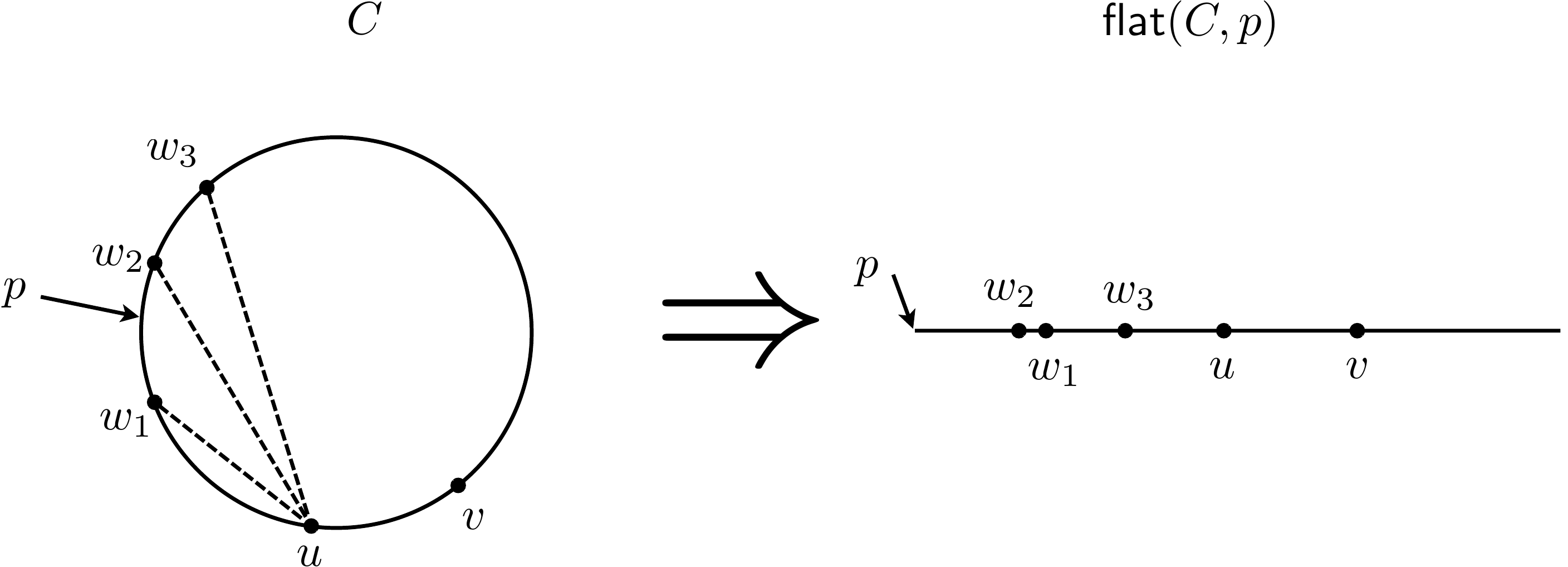}
\caption{If $u$ is not a good vertex for the edge $(u,v)$, then conditions (a) and (b) of Corollary~\ref{cor:choice} imply that $u$ lies between $v$ and $w\in (N_G(u)\cap P_i)\setminus \{v\}$ on the path $\mathsf{flat}(C,p)$.}
\label{fig:flat2}
\end{center}
\end{figure}

Since we choose the anchor points to be $(1/6,1/16)$-apart and $H$ is $160$-slack, by Lemma~\ref{lem:distortion} this embedding is $1$-Lipschitz and has constant distortion.  Thus
we only need to prove the following statements to complete the proof:
\begin{enumerate}
\item Each edge on the outer face of $H_i$ has a good vertex with respect to $V(H_i)$ and $F_i$.
\item This construction produce a star-shaped embedding.
\end{enumerate}

\medskip
\noindent
{\bf Proof of (i).}
Fix a mapping $F_{i-1}:V(H_{i-1})\to V(T_{i-1})$  and let $F_i:V(H_{i})\to V(T_i)$ be the random extension of $F_{i-1}$, where the anchor points are chosen such that they satisfy the conditions (a) and (b) of Corollary~\ref{cor:choice}.
We prove this claim inductively. Suppose that all edges on the outer face of $H_{i-1}$ have a good vertex with respect to $V(H_{i-1})$ and $F_{i-1}$; we show that all edges on the outer face of $H_{i}$ have a good vertex with respect to $V(H_i)$ and $F_i$.
For each edge $(x,y)$  on the outer face of $H_{i}$, we divide the analysis into two main cases.

\medskip
\noindent
{\bf Case I:  $(x,y)\in E(H_{i-1})$.}
Suppose that  $x$ is a a good vertex for the edge $(x,y)$ with respect to $V(H_{i-1})$ and $F_{i-1}$. If $x\notin \{u_i,v_i\}$ or $P^{T_{i-1}}_{F_{i-1}(x)F_{i-1}(y)}\cap P^{T_{i-1}}_{F_{i-1}(u_i)F_{i-1}(v_i)}=\{F_{i-1}(x)\}$  then the inductive hypothesis easily implies that $x$ is also a good vertex for the edge $(x,y)$ with respect to $V(H_i)$ and $F_i$.

If, on the other hand,
$x\in \{u_i,v_i\}$ and $P^{T_{i-1}}_{F_{i-1}(x)F_{i-1}(y)}\cap P^{T_{i-1}}_{F_{i-1}(u_i)F_{i-1}(v_i)}\neq \{F_{i-1}(x)\}$ then, by~\eqref{eq:def:good}  applied to the edge $(x,y)$, $P^{T_{i-1}}_{F_{i-1}(u_i)F_{i-1}(v_i)}$ must be a
subpath of  $P^{T_{i-1}}_{F_{i-1}(x)F_{i-1}(y)}$.
Therefore, by~\eqref{eq:def:good} applied to the edge $(u_i,v_i)$,
the vertex in $\{u_i,v_i\}\setminus \{x\}$ cannot be a good vertex for the edge $(u_i,v_i)$ with respect to $V(H_{i-1})$ and  $F_{i-1}$. Hence, the inductive hypothesis implies that
$x$ is also a good vertex for $(u_i,v_i)$ with respect to $V(H_{i-1})$ and $F_{i-1}$.
%(note that $(x,y)\neq (u,v)$ since $(u,v)$ can not be on the outer face of $ H_{i+1}$). Thus, $x$ could not be a good vertex for the edge $(u,v)$.
In this case for all  $w\in N_G(x)\cap V(P_i)$ by \eqref{eq:lem:slack},  $d_{P_i}(x,w)=d_{H_i}(x,w) \leq (\len(P_i)+d_{H_i}(u_i,v_i))/2$.
Hence conditions (a) and (b) of Corallary~\ref{cor:choice} imply that
\[
P^{T_i}_{F_i(x) F_i(y)}\cap P^{T_i}_{F_i(w) F_i(x)}\subseteq T_{i-1}\cap P^{T_i}_{ F_i(w) F_i(x)}=P^{T_i}_{F_i(u_i) F_i(v_i)}\cap P^{T_i}_{F_i(w) F_i(x)}=\{F_i(x)\}.
\]
 And, in the case that $w\in N_{G}(v)\cap V(H_{i-1})$ since $x$ was a good vertex of $(x,y)$ with respect to $V(H_{i-1})$ and $F_{i-1}$, then $x$ and $w$ also satisfy \eqref{eq:def:good} with respect to $V(H_i)$ and $F_i$.

\medskip
\noindent
{\bf Case II: \bf  $(x,y)\in E(P_i) $.}
%Let $G'_i$ be the  induced subgraph of $G$ on $V(P_i)$, and let $p$ be the anchor point chosen to extend $T$ to $\hat T$. %Note that all the edges of $P_i$ are on the outer face of $G'$.
Lemma~\ref{lem:oneend} implies that each edge $(x,y)$ on the outer face of the subgraph induced by $G$ on $V(P_i)$  (which necessarily contains $E(P_i)$) has a good vertex with respect to $V(P_i)$ and $F_i$.
For an edge $(x,y)\in E(P_i)$ we prove the following statement. If $x$ is a good vertex for the edge $(x,y)$ with respect to $V(P_i)$ and $F_i$ then $x$ is also a good vertex for the edge $(x,y)$ with respect to $V(H_{i})$ and $F_i$.
Let $w\in N_G(x) \cap V(H_i)$. If $w\in V(P_i)$ then the assumption that $x$ is good for $(x,y)$ with respect to $V(P_i)$ implies that it satisfies~\eqref{eq:def:good}.
Therefore, we only need to verify $\eqref{eq:def:good}$ for $w\in N_G(x)\cap V(H_{i-1})$.

If $x\notin\{u_i,v_i\}$ or $P^{ T_i}_{F_i(x)F_i(y)}\cap P^{T_i}_{F_i(u_i)F_i(v_i)}=
\{ F_i(x)\}$ then for all $w\in N_G(x)\cap V(H_{i-1})$, we have $P^{T_i}_{F_i(x)F_i(y)}\cap P^{T_i}_{F_i(x)F_i(w)}\subseteq \{F_i(x)\}$ and~\eqref{eq:def:good} holds.

Now, we consider the case that $x=v_i$ and $P^{T_i}_{F_i(x)F_i(y)}\cap P^{ T_i}_{F_i(u_i)F_i(v_i)}\neq \{ F_i(x)\}$, where $v_i$ is a good vertex for the edge $(u_i,v_i)$ with respect to $V(H_{i-1})$ and $F_{i-1}$. In this case since $x$ is a good vertex for the edge $(x,y)$ with respect to $V(P_i)$ and $F_i$, then  $P^{T_i}_{F_i(u_i)F_i(v_i)}$ must be a subpath of  $P^{T_i}_{F_i(x)F_i(y)}$. Moreover,~\eqref{eq:def:good} for the edge $(u_i,v_i)$ with respect to $V(H_{i-1})$ and $F_{i-1}$, implies that for $w\in N_G(v_i)\cap V(H_{i-1})$, either
 $P^{T_{i}}_{F_{i}(x)F_{i}(w)}\subseteq P^{T_{i}}_{F_{i}(u_i)F_{i}(v_i)}$
in which case
\[P^{T_i}_{F_i(x)F_i(w)} \subseteq P^{T_{i}}_{F_i(u_i)F_i(v_i)} \subseteq P^{ T_i}_{F_i(x)F_i(y)},\]
or $P^{T_i}_{F_i(x)F_i(w)}\cap P^{T_i}_{F_i(u_i)F(v_i)}=\{ F_i(x)\}$ in which case
\begin{multline*}
P^{ T_i}_{F_i(x)F_i(w)}\cap P^{T_i}_{F_i(x)F_i(y)}=
(P^{ T_i}_{F_i(x)F_i(w)}\cap P^{T_i}_{F_i(u_i)F_i(v_i)}) \cup
( P^{ T_i}_{F_i(x)F_i(w)} \cap (P^{T_i}_{F_i(x)F_i(y)} \setminus P^{T_i}_{F_i(u_i)F_i(v_i)}))
\\
\subseteq \{F_i(x)\} \cup ( P^{ T_i}_{F_i(x)F_i(w)} \cap (P^{T_i}_{F_i(x)F_i(y)} \setminus P^{T_i}_{F_i(u_i)F_i(v_i)}))
,\end{multline*}
but the left hand side is a path whereas the right hand side is a disjoint union of a point and a path, hence $P^{ T_i}_{F_i(x)F_i(w)}\cap P^{T_i}_{F_i(x)F_i(y)}=\{F_i(x)\}$.

The final case is where $x=u_i$. In this case by \eqref{eq:lem:slack},  $d_{P_i}(x,y)=d_{H_I}(x,y) \leq (\len(P)+d_{G}(u,v))/2$. Hence by conditions (a) and (b) of Corollary~\ref{cor:choice}, $P^{T_i}_{F_i(x)F_i(y)}\cap P^{T_i}_{F_i(u_i)F_i(v_i)}=\{F_i(u_i)\}$. This case was already dealt with above.
%Therefore, for $w\in N_G(x)\cap V(H_i)$, we have $P^{\hat T}_{xy}\cap P^{\hat T}_{u_iw}\subseteq P^{\hat T}_{xy}\cap P^{\hat T}_{u_iv_i}=\hat F(u_i)$.

%The same argument also implies that all the edges on the outer face of $H_{i+1}$ have a good vertex with respect to $T_R$.
%implies that for $w\in N_{\hat G}()$ either $w\in P_{xy}$ or $P_{xy}\cap P_{wy}=x$. In the case that $y\subseteq V(G_i)\setminus \{u_i\}$, then we do not change the mapping for any of the neighbors of $y$ and the statement holds by our induction hypothesis.

%For this case first note that For any vertex $u'\in P_i\cap N_{G_{i+1}}(u_i)$ by triangle inequality we have
%$d_{G_{i}}(u,u')\leq d_{G_{i}}(u,v)+d_{G_{i}}(v,u')\leq L/2$, where $L$ is the length of the cycle $(u_i,v_i)\cup P_i$.
%Now, suppose that $x=u_i$. We now divide the problem into two cases again. If $y\in P_i$ by the third condition of Lemma~\ref{lem:choice}, $P_T(x,y)\cap P_T(u_i,v_i)=u_i$. Hence, for $w\in N_{G_{i+1}}(y)$ either $w\in P_i$ which Lemma~\ref{lem:oneend} implies $w\in P_T(x,y)$ or $w\in V(G_i)$ which implies $P_T(x,y)\cap P_T(w,x)\subseteq P_T(x,y)\cap P_T(u_i,v_i)=x.$

\medskip
\noindent
{\bf  Proof of (ii).} % We prove this claim by induction. The main idea behind the proof of this claim is the following: In our embedding at the step that we add vertex $x\in P_i$ all of its neighbor get mapped to a path, thus the image of its neighbors is star shaped. After that, at each step when we extend $T$ to trees $L$ and $R$, if we add a neighbor of $w\in N_G(x)$ either $P_{xw}\cap V(T)=x$ or $P_{xw}\subseteq  P_{xy}$ for some $y\in N_G(v)$, and the map remains star shaped. Now we present the proof.
Again, let $F_{i-1}:V(H_{i-1})\to V(T_{i-1})$ a random mapping and let $ F_i:V(H_{i})\to V(T_i)$ be the random extension of $F_{i-1}$.
Note that since $F_i$ is injective, the map is star-shaped if and only if for all $x\in V(H_{i})$, $$\{P^{T_i}_{F_i(x)F_i(y)} : y\in N_G(x) \cap V(H_{i})\}$$ is star-shaped.

We now prove this claim by induction. Suppose that the map $F_{i-1}$ is star shaped; we will show that the map $F_i$ is  also star shaped.
For all vertices $x\in V(H_{i-1})\setminus\{u_i,v_i\}$, by the induction hypothesis,
\[ \{P^{T_{i-1}}_{F_{i-1}(x)F_{i-1}(y)} : y\in N_{G}(x)\cap V(H_{i})\}
=\{P^{T_{i-1}}_{F_{i-1}(x)F_{i-1}(y)} : y\in N_{G}(x)\cap V(H_{i-1})\}\] is star-shaped, and this set remains star-shaped after extension to $F_i$.

For $x\in P_i\setminus \{u_i,v_i\}$, all neighbors of $x$ are mapped to  a path  in
$T_i$, hence $$\{P^{T_i}_{F_i(x)F_i(y)} : y\in N_G(x) \cap V(H_{i})\}$$ is also star-shaped.

Next, suppose that $x=v_i$ is a good vertex for the edge $(u_i,v_i)$ with respect to $V(H_{i-1})$ and $F_{i-1}$.
By the induction hypothesis,
$\{P^{T_{i-1}}_{F_{i-1}(v_i)F_{i-1}(y)} : y\in N_{G}(x)\cap V(H_{i-1})\}$ is star shaped.
Furthermore, since $v_i$  is a good vertex for $(u_i,v_i)$ with respect to $V(H_{i-1})$ and $F_{i-1}$, each  $y\in N_{G}(x)\cap V(H_{i-1})$ is  either on the path  $P^{T_{i-1}}_{F_{i-1}(u_i)F_{i-1}(v_i)}$, or $P^{T_{i-1}}_{F_{i-1}(u_i)F_{i-1}(v_i)}\cap P^{T_{i-1}}_{F_{i-1}(v_i)F_{i-1}(y)}=\{F_{i-1}(v_i)\}$.
Therefore adding the paths $P^{T_i}_{F_i(v_i)F_i(y)}$, for $y\in N_G(v_i)\cap V(P_i)$,
the set $\{P^{T_i}_{F_i(v_i)F_i(y)} : y\in N_{ G}(x)\cap V({H_{i}})\}$ remains star-shaped.

Finally if $x=u_i$, then for  $w\in N_G(x)\cap V(P_i)$ by~\eqref{eq:lem:slack},  $d_{P_i}(x,w)=d_{H_i}(x,w) \leq (\len(P_i)+d_{G}(u_i,v_i))/2$. Hence combining condition (a) and (b) of Corollary~\ref{cor:choice} we can conclude that
\[V(T_{i-1})\cap P^{ T_i}_{F_i(w)F_i(x)}=P^{T_i}_{F_i(u_i)F_i(v_i)}\cap P^{T_i}_{F_i(w)F_i(x)}=\{ F_i(x)\}.\]
In other words, all $G$-neighbors of $x$ in $P_i$ are mapped to a new branch in $T_i$ that  intersects $T_{i-1}$ only at $F_i(x)$. Thus
$\{P^{T_i}_{F_i(x)F_i(y)} : y\in N_G(x) \cap V(H_{i})\}$ is also star-shaped.
\end{proof}

\section{Connected random retractions}
\label{sec:retract}

Our goal now is to complete the proof of Theorem
\ref{thm:main} by showing that every planar graph
can be randomly retracted onto a specified face
in such a way that the face can itself be endowed with an outerplanar metric.
Combining this with our embedding of outerplanar graphs into random trees
from Section \ref{sec:outerplanar}, we will be able to prove Theorem \ref{thm:main};
this is done in Section \ref{sec:retractOP}.

In the next section, we review the notion of ``padded partitions'' of metric spaces.
The existence of such partitions for planar graphs (due to \cite{KPR93}) will be
one of our two central ingredients here.  The other ingredient is the method of
\cite{EGKRTT10} for the construction of random {\em connected} retractions.
They work with a weaker notion of random partitions, so their results
(as stated in \cite{EGKRTT10}) are not strong enough for us.  In Section \ref{sec:retracts},
we follow their proof closely but use padded partitions,
allowing us to obtain the stronger conclusion we require.

\subsection{Padded partitions of graphs}

Random partitions are a powerful tool in the theory of embeddings
of finite metric spaces; see, e.g.,  \cite{Bartal98,Rao99,KLMN04,LN05}.
A particularly powerful notion is that of a ``padded'' partition.
We review the relevant definitions in the special setting of finite metric spaces.

Consider a metric space $(X,d)$.  We will sometimes think of a partition $P$ of $X$ as a map
$P : X \to 2^X$ sending each $x \in X$ to the unique set in $P$ containing it.  We say that $P$ is
{\em $\tau$-bounded} if $\diam(S) \leq \tau$ for every $S \in P$.
We say that a random partition $\mathcal P$ is $\tau$-bounded if this holds
almost surely.  A random partition $\mathcal P$ is {\em $(\alpha,\tau)$-padded} if
it is $\tau$-bounded and, additionally, for every $x \in X$ and $R \geq 0$, we have
$$
\pr[B_X(x,R) \nsubseteq \mathcal P(x)] \leq \alpha \cdot \frac{R}{\tau}\,,
$$
where $B_X(x,R) = \{ y \in X : d(x,y) \leq R \}$.

The main random partitioning result we require is from \cite{KPR93}, though it first appeared in this form later (see \cite{Rao99,KLMN04,LN05}).

\begin{theorem}[\cite{KPR93}]
\label{thm:kpr}
There exists a constant $\alpha > 0$ such that
if $G=(V,E)$ is a metric planar graph, then for every $\tau > 0$,
$(V,d_G)$ admits an $(\alpha,\tau)$-padded random partition.
Furthermore, the distribution of the partition can be sampled
from in polynomial-time in the size of $G$.
\end{theorem}

\subsection{Random retractions}
\label{sec:retracts}

We now use random partitions to construct random retractions.  This was first done in \cite{CKR01} in the
context of the 0-extension problem on graphs.  Further work includes \cite{LN05},
which concerns the Lipschitz extension problem, and
\cite{LS09}, where the authors are primarily concerned with randomly simplifying the topology of metric graphs.
The proof of the next theorem follows from the techniques of \cite{EGKRTT10} for constructing a connected retraction.
We are able to obtain a stronger conclusion by using a stronger assumption about the random partitions.

\begin{theorem}\label{thm:connretract}
Let $G=(V,E)$ be a metric graph and suppose that for some $\alpha \geq 2$ and every $\tau \geq 0$,
$(V,d_G)$ admits an $(\alpha,\tau)$-padded random partition.
Then for any subset $S \subseteq V$, there exists a random mapping $F : V \to S$ such that the following properties hold.
\begin{enumerate}
\item For every $x \in S$, $F(x)=x$.
\item For every $x \in V$ and $\tau > 0$, $\E\,|\nabla_{\tau} F(x)|_{\infty} \leq O(\alpha\log \alpha)$.
\item For every $x \in S$, the set $F^{-1}(x)$ is a connected subset of $G$.
\end{enumerate}
\end{theorem}

\begin{proof}
Since we are dealing with finite graphs,
without loss of generality, we may assume that $d_G(S, V \setminus S) > 1$.
Let $k_0 = \lceil \log_2 \diam_G(V) \rceil$.
First, we let $\{P_k : 1 \leq k < k_0 \}$ be a sequence of independent $(\alpha,2^k)$-padded
random partitions.
We define $P_{k_0}=\{V\}$.
Following \cite{EGKRTT10}, we inductively
define a sequence of random maps $\{F_k : V_k \to S : 1 \leq k \leq k_0 \}$,
where $V_k \subseteq V$ and $(F_{k+1})|_{V_k} = F_k$ for each $k \in \mathbb N$.
%Also, let $k_0 = \lceil \log_2 \diam_G(V) \rceil$.
We will define $F = F_{k_0}$.

First, we put $V_0 = S$ and $F_0(x)=x$ for $x \in S$.
Now suppose that $V_{k-1}$ and $F_{k-1}$ are defined
for some $k \geq 1$.
We will use the notation $\hat P_k(x)$ for the connected component of
$G[P_k(x)]$ containing $x$.   We let
$$
V_k = \{ x \in V : \hat P_k(x)\cap V_{k-1} \neq \emptyset \}\,
$$
be the set of vertices which are connected to $V_{k-1}$ through their set in $P_k$.

For every $T \in P_k$ and every non-empty connected component $C$ of $G[T \cap (V_k \setminus V_{k-1})]$,
let $v_C$ be a neighbor of $C$ in $V_{k-1}\cap P_k(x)$, which must exist by the definition of $V_k$.
For $x \in C$, we define $F_k(x)=F_{k-1}(v_C)$.  For all other $x \in V_k$,
we must have $x \in V_{k-1}$, and we put $F_k(x)=F_{k-1}(x)$.
By definition, we have $V_{k_0}=V$ and we define $F=F_{k_0}$.

By construction, property (i) holds.
Property (iii) follows easily by induction:  It is true for every $x \in S$ and $0 \leq k \leq k_0$ that $F_k^{-1}(x)$ is connected.
We are thus left to verify property (ii).
For every vertex $x \in V$, define $L(x) = \min \{k : x \in V_k \}$.  We make
the following claim.

\begin{comment}
\begin{claim}\label{claim:tech}
The following assertions hold true.
\begin{enumerate}
\item[(1)] For every $x \in V$, we have $d_G(x,F(x)) \leq 2^{L(x)+1}$.
\item[(2)] For every $x \in V$, $2^{L(x)+1} > d_G(x,S)$.
\end{enumerate}
\end{claim}

\begin{proof}
Claim (1) follows by induction:  For any $0 \leq k \leq k_0$, we claim that if $x \in V_k$, then
$d_G(x,F_k(x)) \leq 2^{k+1}$.  This is clear for $L(x)=0$ since $V_0=S$ and $F_0(x)=x$ for all $x \in S$.
If $L(x)=k > 0$, then $x \in V_{k} \setminus V_{k-1}$, hence $F_k(x)=F_{k-1}(y)$ for some $y \in P_k(x) \cap V_{k-1}$.
Thus, $$d_G(x,F_k(x)) \leq \diam(P_k(x)) + d_G(y,F_{k-1}(y))\,.$$
Since $P_k$ is $2^k$-bounded, we have $\diam(P_k(x)) \leq 2^k$ and by
induction, $d_G(y,F_{k-1}(y)) \leq 2^{k}$.  It follows that $d_G(x,F_k(x)) \leq 2^{k+1}$, completing the proof.

Claim (2) also follows by induction.  If $L(x)=0$ then it is obvious.
If $L(x)=k > 0$, then $P_k(x) \cap V_{k-1}$ is non-empty, meaning that
$$
d_G(x,S) \leq \diam(P_k(x)) + \min_{y \in P_k(x) \cap V_{k-1}} d_G(y,S) < 2^k + 2^k = 2^{k+1} = 2^{L(x)+1}\,,
$$
where we have used the induction hypothesis in the final inequality.
\end{proof}
\end{comment}

\begin{claim}\label{claim:tech}
For every $x \in V$,  $d_G(x,S) \le d_G(x,F(x)) < 2^{L(x)+1}$
\end{claim}

\begin{proof}
The first inequality is immediate since $F(x)\in S$. The second
 follows by induction:  For any $0 \leq k \leq k_0$, we claim that if $x \in V_k$, then
$d_G(x,F_k(x)) < 2^{k+1}$.  This is clear for $L(x)=0$ since $V_0=S$ and $F_0(x)=x$ for all $x \in S$.
If $L(x)=k > 0$, then $x \in V_{k} \setminus V_{k-1}$, hence $F_k(x)=F_{k-1}(y)$ for some $y \in P_k(x) \cap V_{k-1}$.
Thus, $$d_G(x,F_k(x)) \leq \diam(P_k(x)) + d_G(y,F_{k-1}(y))\,.$$
Since $P_k$ is $2^k$-bounded, we have $\diam(P_k(x)) \leq 2^k$ and by
induction, $d_G(y,F_{k-1}(y)) < 2^{k}$.  It follows that $d_G(x,F_k(x)) < 2^{k+1}$, completing the proof.
\end{proof}

\noindent
Now fix $x \in V$ and $\tau > 0$ and let $B = B_G(x,2\tau)$.
We will employ the bound
\begin{equation}\label{eq:stretch}
|\nabla_{\tau} F(x)|_{\infty} \leq \frac{\diam_G(F(B))}{\tau}\,.
\end{equation}
Let $\maxB = \max \{ L(x) : x \in B \}$ and $\minB = \min \{ L(x) : x \in B \}$.
Using the triangle inequality and Claim \ref{claim:tech}, we have
\begin{equation*}
\diam_G(F(B)) \leq 2^{\maxB + 2} + 4 \tau\,.
\end{equation*}

%By Claim \ref{claim:tech}(2), we have $2^{\min_B} \geq d_G(x,S)-2 \tau$.
Let $\mathcal E_k$ be the event that $B \nsubseteq P_k(x)$.
%Put $\mathcal E = \bigwedge_{k=\min_B}^{\max_B} (\lnot \mathcal E_k)$.
Observe that $$\lnot \mathcal E_{\minB} \implies |F(B)|=1 \implies \diam_G(F(B))=0,$$
because in this case, every vertex in $B$ is in the same connected component of $x$ in $G[P_k(x) \setminus V_{k-1}]$,
where $k = \minB$ (they are all connected through $x$).
Hence,
\begin{equation}\label{eq:local}
\diam_G(F(B)) \leq 4\tau + \1_{\{ \mathcal E_{\minB}\}} \cdot 2^{\maxB + 2}\,.
\end{equation}
Let $\Delta= \min\{d_G(y,S):y\in B\}$, and $m=\lceil \log_2 \max(1,\Delta)\rceil$. By Claim \ref{claim:tech}, we have $\minB\geq m-1$, hence
\begin{eqnarray}
\E \left[  \1_{\{ \mathcal E_{\minB}\}} \cdot 2^{\maxB + 2}\right] &=& \sum_{k\geq m-1} \pr(k=\minB, \mathcal{E}_{k}) \cdot \E\left[2^{\maxB+2} \,\Big|\, \mathcal{E}_k, \minB= k\right]
%\\
%&\leq &
%\sum_{k \geq m} \pr(\mathcal{E}_{k}) \pr(\maxB \geq k) \cdot \E\left[2^{\max_B+2} \,\Big|\, \mathcal E_k, \textrm{max}_B \geq k\right]\,,
\label{eq:summing}
\end{eqnarray}
%where we have used
%$$\pr(\mathcal{E}_{k}, \textrm{max}_B \geq k) = \pr(\mathcal E_k) \pr(\textrm{max}_B \geq k \mid \mathcal E_k) \leq
%\pr(\mathcal E_k) \pr(\textrm{max}_B \geq k-1 \mid \mathcal E_k) =\pr(\mathcal{E}_{k}) \pr(\textrm{max}_B \geq k-1),$$
Now, observe that, for $k\geq \minB$,
 $$\lnot \mathcal E_{k} \implies \maxB\leq k,$$
because, in this case, $P_k(x)$ must contain a vertex $v \in V_\minB$ and a shortest-path from $v$ to $x$,
implying that $B_G(x,2\tau) \subseteq V_k$.
The padding property implies that
$$
\pr\left(\mathcal E_j \right) \leq {2\alpha\tau\over 2^j}\,.
$$
In particular, for any $k \geq \max(\minB, \log_2 2\alpha \tau)$, we have
\begin{equation}\label{eq:maxeq}
\pr(\maxB \geq k+j \mid \mathcal E_k, \minB = k) \leq \prod_{i=k+1}^{k+j-1}\pr\left(\mathcal E_i \right) \leq \prod_{i=k+1}^{k+j-1}2^{k-i} \leq 2^{-(2j-3)}\,.
\end{equation}
Using this yields
\begin{equation}\label{eq:maxB}
\E\left[2^{\maxB+2} \,\Big|\, \mathcal E_k, \minB = k\right] \leq  4\cdot \max(2^{k}, 2\alpha\tau) \sum_{j=0}^{\infty} 2^{-(2j-3)}\cdot 2^{j} \leq 2^{k+6}+128\alpha \tau\,.
\end{equation}
%Since  $\mathcal E_k$ only depends
%on the independent random variable $P_k$, by the padded property (b), we have
%\begin{equation}\label{eq:onelev}
%\pr(\mathcal E_k) \leq \alpha \frac{2\tau}{2^k}\,.
%\end{equation}
Combining \eqref{eq:summing} and \eqref{eq:maxB}, we can write
\begin{eqnarray}
\nonumber\E \left[  \1_{\{\mathcal E_\minB\}} \cdot 2^{\maxB + 2} \right] &\leq& \sum_{k \geq m-1} \pr(k=\minB, \mathcal E_k) \cdot (2^{k+6}+128\alpha \tau)\\
\nonumber&\leq& 128\alpha \tau+ \sum_{k \geq m-1} \pr(k= \minB,\mathcal E_k)2^{k+6}\\
&\leq& 128\alpha \tau+ \sum_{k \geq m-1} \pr(k\leq  \minB)\pr(\mathcal E_k)2^{k+6},\label{eq:summing2}
\end{eqnarray}
where the last inequality holds because
\[ \pr(k= \minB,\mathcal E_k)\leq \pr(k-1< \minB,\mathcal E_k) = \pr(\minB>k-1)\cdot \pr(\mathcal E_k \mid \minB>k-1) =\pr(\minB>k-1)\cdot \pr(\mathcal E_k).
\]
Let $y\in B$ and $w\in S$ be a pair of points such that $d_G(y,w)=\Delta$. It is easy to check that if for some $k$, $B_G(w,\Delta)\subseteq P_k(w)$ then $\minB\leq k$,
thus the properties of padded partitions imply that
\[
\pr(\minB\geq k)\leq \prod_{i<k} \pr(B(w,\Delta)\nsubseteq P_i(w))\leq \prod_{i<k} \min\left(1,{2\alpha 2^m\over 2^i}\right)\leq \min\left(1,{2\alpha 2^m\over 2^{k-1}}\right)
\]
Combining this with \eqref{eq:summing2}, we can conclude that
\begin{eqnarray*}
\E \left[  \1_{\{\mathcal E_\minB\}} \cdot 2^{\maxB + 2} \right]
&\leq& 128\alpha \tau+ \sum_{k \geq m-1} \min\left(1, {2\alpha 2^m\over 2^{k-1}}\right)\left(\frac{2\alpha\tau} {2^k}\right) 2^{k+6}\\
&=& 128\alpha \tau+ \sum_{k \geq m-1} \min\left(1, {2\alpha 2^m\over 2^k}\right)(128\alpha \tau)\\
&=& 128\alpha \tau+\sum_{k=m-1}^{ m+\lceil \log_2 \alpha\rceil } 128\alpha \tau+\sum_{k > m+\lceil \log_2 \alpha\rceil } \left({2\alpha 2^m\over 2^k}\right)(128\alpha \tau)\\
&\leq &O(\alpha\tau)+O\left(\alpha \tau \log \alpha\right)+O(\alpha\tau) \\
&\leq & O\left((\alpha\log \alpha) \tau\right).
\end{eqnarray*}
%w here in the penultimate line we have used \eqref{eq:maxeq} to write $\pr(\mathrm{max}_B \geq k-1) \leq 2^{m+1-k}$.

Combining this with \eqref{eq:stretch} and \eqref{eq:local} completes the verification of property (ii).
\end{proof}

\subsection{Retracting to an outerplanar graph}
\label{sec:retractOP}

Finally, we use the random retractions of the preceding section
to randomly embed every metric on the face of a planar
graph into an outerplanar graph in a suitable way.
This technique is also taken from \cite{EGKRTT10},
although again we require some stronger properties
of the embedding.

\begin{theorem}\label{thm:ontoface}
There is a constant $K > 1$ such that
for any metric planar graph $G=(V,E)$ and face $V_0 \subseteq V$,
there is a random outerplanar metric graph $H$ and a random mapping $F : V\to V(H)$
satisfying the following:
\begin{enumerate}
\item For every edge $\{u,v\} \in E$, either $F(u)=F(v)$ or $\{F(u),F(v)\} \in E(H)$.
\item For every $u,v \in V_0$, $d_{H}(F(u),F(v)) \geq d_G(u,v)$.
\item For every $u \in V$ and $\tau \geq 0$, we have $\E\,|\nabla_{\tau} F(u)|_{\infty} \leq K$.
\end{enumerate}
\end{theorem}

\begin{proof}
By Theorem \ref{thm:kpr}, we can
apply Theorem \ref{thm:connretract}
to the metric graph $G$ with $S=V_0$.
Let $F : V \to V_0$
be the random mapping guaranteed by Theorem \ref{thm:connretract}.

We construct
a metric graph $H$ with vertex set $V_0$ and an edge $\{u,v\}$ of length $d_G(u,v)$
whenever there is an edge between the sets $F^{-1}(u)$ and $F^{-1}(v)$ in $G$.
Since the sets $\{F^{-1}(u) : u \in V_0\}$ are connected, the resulting graph $H$
is outerplanar.  Also, property (i) is immediate.

Property (ii) follows because $F$ is the identity on $V_0$
and the edges in $H$ have length equal to the distance between their endpoints in $(V_0, d_G)$.
Property (iii) follows from Theorem \ref{thm:connretract}(i).
\end{proof}

We can now complete the proof of Theorem \ref{thm:main}.

\begin{proof}[Proof of Theorem \ref{thm:main}]
Let $\Lambda_1 : V \to V(H)$ be the random mapping
from $V$ onto the vertices of an outerplanar metric graph $H$
guaranteed from Theorem \ref{thm:ontoface}.
Let $\Lambda_2 : V(H) \to V(T)$ be the random mapping of $H$ into trees from Theorem \ref{thm:tree-embed}.
The mapping $\Lambda= \Lambda_2 \circ \Lambda_1 :V\to V(T)$ is mapping guaranteed by the theorem.
Combining the star-shaped property of Theorem \ref{thm:tree-embed} with Theorem \ref{thm:ontoface}(i) implies
that $\Lambda$ is star-shaped.  Property (ii) of Theorem \ref{thm:main} is
a consequence of Theorem \ref{thm:ontoface}(ii) and Theorem \ref{thm:tree-embed}.
Finally, property (i) follows from the Lipschitz condition of Theorem \ref{thm:tree-embed}
and property (iii) of Theorem \ref{thm:ontoface}.
\end{proof}

\subsection*{Acknowledgements}

We are grateful to Chandra Chekuri and Bruce Shepherd for bringing the node-capacitated
problem to our attention and for their encouragement during the course of the project.

\bibliographystyle{alpha}
\bibliography{nodeoks-mm}

\end{document}